\documentclass[11pt]{article}

\usepackage{latexsym,color,graphicx,amsmath,amssymb,amsthm,upquote,commath,enumerate,titlefoot}
\usepackage[hidelinks]{hyperref}

\def\g{{\mathfrak{g}}}
\def\a{{\mathfrak{a}}}
\def\R{{\mathbb R}}
\def\G{{\mathbb G}}
\def\Q{{\mathbb Q}}
\def\N{{\mathbb N}}

\def\C{{\mathbb C}}

\def\q{{\mathfrak q}}

\def\y{{\bf y}}

\def\Z{{\mathbb Z}}

\def\E{{\rm{E}}}

\def\SL{{\rm{SL}}}
\def\Aff{{\rm{Aff}}}
\def\Diff{{\rm{Diff}}}
\def\SR{{\rm{SR}}}

\def\Sp{{\rm{Sp}}}
\def\GL{{\rm{GL}}}

\newtheorem{thm}{Theorem}[section]

\newtheorem{cor}[thm]{Corollary}
\newtheorem{con}[thm]{Construction}
\newtheorem{lem}[thm]{Lemma}
\newtheorem{ex}[thm]{Example}
\newtheorem{rem}[thm]{Remark}

\newtheorem{prop}[thm]{Proposition}
\theoremstyle{definition}
\newtheorem{defin}[thm]{Definition}

\begin{document}
\author{Leonid Polterovich\and  Yehuda Shalom\and Zvi Shem-Tov}
\title{Norm rigidity for arithmetic and profinite groups}
\date{}
\maketitle
\unmarkedfntext{2020 Mathematics Subject Classification: 22E40; 22XX; 20E18; 20F65}
\abstract{Let $A$ be a commutative ring, and assume every non-trivial ideal of $A$ has finite-index.
We show that if $\SL_n(A)$ has bounded elementary generation then every conjugation-invariant norm on it
is either discrete or precompact. If $G$ is any group satisfying this dichotomy we say that $G$ has the
\emph{dichotomy property}. We relate the dichotomy property, as well as some natural variants of it, to
other rigidity results in the theory of arithmetic and profinite groups such as the celebrated normal
subgroup theorem of Margulis and the seminal work of Nikolov and Segal. As a consequence
we derive constraints to the possible approximations of certain non residually finite
central extensions of arithmetic groups, which we hope might have further applications in the study of sofic groups.
In the last section we provide
several open problems for further research.
}

\section{Introduction}
A classical theorem of Ostrowski says that every absolute value on the field of rational numbers $\Q$,
or equivalently on the ring of integers $\Z$, is equivalent to either the standard (real) absolute value, or
a $p$-adic absolute value for which the closure of $\Z$ is compact.
One of the main purposes of this paper is a non-abelian analogue of this result for $\SL_{n\ge3}(\Z)$ and
related groups of arithmetic type. In a different direction, our results can also be viewed as part of the Margulis--Zimmer
rigidity theory: they extend the celebrated Margulis' normal subgroup theorem for some arithmetic groups, and
prove for them rigidity theorems for homomorphisms into certain non-locally compact groups-- those equipped with a
bi-invariant metric. Somewhat surprisingly, this turns out to be strongly related also to the deep work of Nikolov--Segal
on profinite groups, and to constraints on sofic approximations of certain central extensions of arithmetic groups.
To make all this precise, we shall need the notion of a conjugation-invariant norm on a group.

\paragraph {Norms on groups and some basic constructions.}
A \emph{conjugation-invariant norm} on a group $G$ 
(or simply, a \emph{norm} \footnote{There is also a notion of a norm on a group which is not conjugation invariant, but in this paper we only consider conjugation-invariant norms and occasionally omit the adjective ``conjugation-invariant''. }) is a function $\norm{\cdot}:G\to \R$ satisfying
the following properties:
\begin{enumerate}[(i)]
\item $\norm{g}\ge0$ for every $g\in G$
\item $\norm{g}=0$ if and only if $g=1_G$
\item $\norm{g^{-1}}=\norm{g}$ for every $g\in G$
\item $\norm{gh}\le\norm{g}+\norm{h}$ for every $g,h\in G$
\item $\norm{hgh^{-1}}=\norm{g}$ for every $g,h\in G$
\end{enumerate}

The most obvious construction of a conjugation-invariant norm on a group is
 a \emph{Dirac norm}, assigning $0$ to the identity element and $1$ to any other element.
 This norm can easily be perturbed, for example by replacing $1$ by any other value, or changing its value slightly so that
 it is still constant on conjugacy classes and satisfies the triangle inequality thus creating a family of what we call the
 \emph{discrete norms} -- see below.
  A less obvious, yet still extremely natural construction of a norm on a group $G$ comes from embedding it in a
  compact group $K$. It is well known (\cite{KL}) that any compact metrizable group admits a bi-invariant metric $d$, and it is easily seen that setting
  $\norm{g} = d(1,g)$ defines then a norm on $G$ inherited by the embedding. We call below such norms \emph{precompact}.
   If the group $G$ is residually finite then
  such $K$ can always be taken as the \emph{profinite completion} of $G$.\footnote{Recall that a group $G$ is called \emph{residually-finite} if
the intersection of all of its finite-index (normal) subgroups is trivial.}
For the benefit of non specialists let us make this construction concrete:

\begin{con}[The norm corresponding to a chain of finite-index subgroups]\label{con}
Let $G$ be a residually-finite group. Let
$$
G=G_0\rhd G_1\rhd\dots
$$
be a descending chain of finite-index normal subgroups of $G$ with trivial intersection.
For each element $g$ in $G$ set $\norm{g}:=2^{-i}$ if $g$ belongs to $G_i- G_{i+1}$.
It is easy to see that $\norm{\cdot}$ is a conjugation-invariant norm on $G$, and that this
norm is precompact. We say that $\norm{\cdot}$ is the \emph{standard norm on $G$
corresponding to the sequence $(G_i)$}. Note that the metric completion of any such norm is profinite
\footnote{Recall that a profinite group is a Hausdorff, compact group, which
is totally disconnected, i.e. the connected components are singletons.}.
\end{con}
We remark that the specific choice of the sequence $2^{-i}$ above is immaterial; we could take any decreasing sequence
instead. We remark that for finitely generated groups the connection between residual finiteness and the existence of pre-compact norms
goes in the other way as well: if the group admits such norm, apply Peter-Weyl Theorem and the well known fact that finitely generated linear groups
are residually finite.

One can further combine the above construction with that of the Dirac norm to get the following:

\begin{con}[Singular extension of a conjugation-invariant norm]\label{con2}
Let $G$ be a group, and $N$ a normal subgroup of $G$. Suppose that $\norm{\cdot}_N$ is a
conjugation-invariant norm on $N$, satisfying $\norm{n}_N\le1$ for all $n\in N$.
Suppose $\norm{gng^{-1}}_N=\norm{n}_N$ for each
$g\in G$. Define a function on $G$ by
\begin{equation}\label{eq1}
\norm{g}:=
\begin{cases}
\norm{g}_N& g\in N\\
1& g\notin N.
\end{cases}
\end{equation}
It is easy to see that $\norm{\cdot}$ is a conjugation-invariant norm on $G$.
We say that $\norm{\cdot}$ is the \emph{singular extension} of $\norm{\cdot}_N$ to $G$.
\end{con}

\noindent {\bf Some key definitions and notions.} Generalizing our previous terminology, we say that a conjugation-invariant norm on $G$ is \emph{discrete} if it generates the discrete topology on $G$,
and \emph{compact (precompact)} if every sequence has a converging (resp. Cauchy) subsequence.
Our starting point is the work of Burago, Ivanov and Polterovich \cite{BIP}, who studied conjugation-invariant norms on groups of diffeomorphisms of smooth manifolds. If $M$ is a smooth closed connected manifold, they showed that any conjugation-invariant norm on $\Diff_0(M)$ (the connected component of the identity of the diffeomorphism group of $M$) is discrete. Furthermore, this group is bounded for all closed manifolds $M$ provided $\dim M=1,3$ (see \cite{BIP}) and $\dim M \geq 5$ (see \cite{Ts1,Ts2}). In dimension $2$ this group is bounded for the sphere \cite{BIP} and, remarkably, it is unbounded for oriented surfaces of positive genus \cite{BHW}. In the remaining cases (non-orientable surfaces and $\dim =4$), boundedness of $\Diff_0(M)$ is unknown.

An abstract group $G$ is called \emph{meager}, if any conjugation-invariant norm on it is both bounded and discrete. Hence in this
terminology many diffeomorphism groups are meager.

The main purpose of the present paper is to establish rigidity results for norms on special linear groups.
We remark that a different phenomenon, that of
\emph{boundedness} (of all norms) in arithmetic groups, was extensively studied (see e.g. \cite{Ked}, \cite{Trost} and references therein).
All these works make use of \emph{bounded generation}, and so shall we, though
by completely different means. After recalling below the definition of bounded generation, we shall
make the interesting observation that boundedness of all norms on $\E_n(A)$
is in fact \emph{equivalent} to it.

The following is our key definition of the paper.
\begin{defin}\label{dp}
Let $G$ be a group. We say that $G$ has the \emph{dichotomy property} if every conjugation-invariant
norm on it is either discrete or precompact.
\end{defin}

\begin{rem}\label{rem-bdd} One readily checks that for every norm $\norm{x}$ on $G$,
the expression $\norm{x}' := \norm{x}/(1+\norm{x})$ defines another, \emph{bounded} norm on $G$,
for which the properties of discreteness, compactness or precompactness remain unchanged.
Thus, for verification of the dichotomy property
it suffices to stick to bounded norms only.
\end{rem}

One of our main motivations in defining the dichotomy property is the following
simple result, connecting it to the famous Normal Subgroup Theorem of Margulis for higher rank arithmetic groups:
\begin{prop}\label{nsg}
Let $G$ be a residually finite group. Assume that $G$ has the dichotomy property. Let $N$ be a normal
subgroup of $G$. Then either $N$ has finite index in $G$ or $N$ is finite.
\end{prop}
As remarked above, the existence of a precompact norm for finitely generated groups implies residual finiteness, hence assuming the latter property is very natural in this context. As the proof of the proposition is quite simple, we can bring it here:

\begin{proof}
Since $G$ is residually finite, there exists a descending chain $N_0\rhd N_1\rhd\dots$ of finite-index subgroups of $N$, which are all normal in $G$
and intersect trivially. Let $\norm{\cdot}_N$ be the profinite norm on $N$ corresponding to $N_i$
(Construction \ref{con}). Let $\norm{\cdot}$ be the singular extension of $\norm{\cdot}_N$ to $G$
(Construction \ref{con2}).
Suppose that $N$ is infinite. Then clearly $\norm{\cdot}_N$ is non-discrete. Thus
$\norm{\cdot}$ is non-discrete as well. Thus
$\norm{\cdot}$ is precompact.
Combining this with the fact that $\norm{g}=1$ for each $g$ in the complement $G-N$,
we get that $N$ must be of finite-index in $G$.
\end{proof}

\begin{ex}\label{ex-Z-1} {\rm The infinite cyclic group $\Z$ is just infinite but does not satisfy the dichotomy property: a (tricky!) example of a norm which is neither discrete, nor precompact is constructed in \cite[Theorem 24]{Nienhuys}. More generally, one can use this fact and a slight modification of the argument above
to deduce that any group $G$ having a non-torsion central element does not have the dichotomy property (distinguish here between the cases where the cyclic group
generated by this element has finite or infinite index). This result can then be easily generalized further to any group $G$ with infinite center.}
\end{ex}

The main class of groups discussed in this paper is $\SL_n(A)$ for various rings $A$. In all examples
known to us, if a group of this kind satisfies the dichotomy property,
then the metric completion of any non-discrete conjugation-invariant norm on it is in fact profinite. Hence we
make the following definition.
\begin{defin}\label{sdp}
A group $G$ satisfies the \emph{strong dichotomy property} if the metric completion of any non-discrete conjugation-invariant
norm on $G$ is profinite.
\end{defin}

\paragraph {The main results -- discrete setting.} To present our main result, we return to the notion of bounded generation, extending it further
in the context of special linear groups over general rings. Fix an integer $n\ge2$ and let $A$ be an integral domain with unit.
Let $\E_n(A)$ be the subgroup of $\SL_n(A)$ generated by all the elementary matrices.
We say that this group has \emph{bounded elementary generation} if there exists some constant $C$ such that every matrix
in  $\E_n(A)$ is a product of at most $C$ elementary matrices.
Restricting from now on to $n\ge3$, notice that \emph{boundedness} (of all norms on this group)  is not only implied by its being boundedly
elementary generated (cf. \cite{Ked} and references therein), but
also \emph{implies} it:  This follows immediately by considering the
conjugation-invariant norm on $G$ assigning to $g$ the minimal integer $l$ for which $g$ can be written
as a product of $l$ conjugates of elementary matrices, and using the Suslin factorization theorem (see e.g. {\cite[Corollary 5.2]{Stepanov}})
saying that any such conjugate is
a product of a uniformly bounded number of actual elementary matrices. Thus, for rigidity phenomena related
to norms, bounded elementary generation does in fact arise as an extremely natural concept.

Our main result is the following theorem.

\begin{thm}\label{mainthm}
Keep the above notations and
assume that any non-trivial ideal $\q$ in $A$ is
of finite-index in $A$. Then
\begin{enumerate}[(i)]
\item \label{main11}
If $\E_n(A)$ has bounded elementary generation, then
$\E_n(A)$ has the strong dichotomy property.
\item
\label{main21}
If $\E_n(A)$ has the (strong) dichotomy property, then
every finite-index subgroup of $\E_n(A)$ has the (strong) dichotomy property as well.
\end{enumerate}
\end{thm}
We prove the theorem in Section \ref{sec1}.
Examples of rings $A$ satisfying the condition in the theorem include $A=\mathcal{O}_K$ -- the integer ring
of number field $K$ (in particular $A=\Z$, if $K=\Q$), $A=\Z[\frac{1}{p}]$ and $A=\mathbb{F}_q[x]$ -- the polynomial ring in one variable over a finite field (or more generally, rings of all $S$-algebraic integers in general global function fields \cite{trost2}). 
While the theorem is formulated for $\E_n(A)$, in all of these examples, we actually have
$\E_n(A)=\SL_n(A)$. In fact, if $\E_n(A)$ has the dichotomy property, then $\SL_n(A)$ has the dichotomy
property if and only if $\E_n(A)$ has finite index in $\SL_n(A)$, as follows immediately from Proposition
\ref{nsg} and Suslin's normality theorem (see e.g. {\cite[Corollary 5.2]{Stepanov}}) and Lemma \ref{up} below.
Returning to the opening paragraph of the introduction, in the case of $A=\Z$, it can be deduced from the congruence subgroup 
property (CSP), that any profinite completion of $\SL_{n\ge3}(\Z)$, is a quotient of $\SL_n(\hat{\Z})$. Therefore just as in 
Ostrowski's theorem any norm completion is of well-understood arithmetic type.

\begin{rem}

\begin{enumerate}
\item The requirement that every non-zero ideal of $A$ has finite-index is necessary. Indeed,
if $A$ is a commutative ring with unit and $\q$ is a non-zero ideal in $A$ of infinite index, then consider
the natural map $\pi:\E_n(A)\to \E_n(A/\q)$. The image and the kernel of $\pi$ are infinite, hence
$\E_n(A)$ does not have the dichotomy property, due to Proposition \ref{nsg}. To construct such a ring $A$
for which $\SL_n(A)$ has elementary bounded generation, start with any principal ideal domain $A'$ having a non-zero 
prime ideal $\q$ of infinite index (e.g. $A'=\Q[x],\q=xA'$) and set $A$ as the localization of $A'$ at $\q$.
Then $A$ is a discrete valuation ring, $\q$ still has infinite index when viewing it as an ideal of $A$, and 
since $A$ is a local ring, it then follows easily from the proof of \cite[Theorem 4.1]{Sh} that $\E_n(A)$ has bounded elementary generation.  

\item We do not know if the requirement $n\ge3$ can be relaxed to $n\ge2$. See \cite{Trost3} for a positive result in this direction. 
\end{enumerate}
\end{rem}
The idea of the proof of the first part of the theorem
is showing that any non-discrete conjugation invariant norm induces a non-discrete precompact norm on
each elementary subgroup, then using bounded generation to obtain precompactness of the original norm.
{\it The proof of the second part is more involved, as in general, the dichotomy
property is not preserved under passage to finite-index subgroups} (Proposition \ref{example}).
To prove this second part, we introduce a somewhat weaker version of the dichotomy property,
called the \emph{weak} dichotomy property, which passes to finite-index normal subgroups, and
turns out to be equivalent to the dichotomy property in our situation.
The case of $A=\mathbb{F}_q[x]$ where one has bounded elementary generation for the full Linear group \cite{Nica},
but nothing seems to be known about finite index (congruence) subgroups of it, shows a natural example where
this approach seems to become necessary.

\paragraph{The main results -- dual compact setting.} We next move to present a duality between the discrete and compact framework.
While in our previous discussion
we started with a discrete group (i.e. no topology on the group is specified), and the dichotomy property ensured that any non-discrete norm on it arises
from the compact setting, we now want to discuss the opposite situation, starting from a compact (metrizable) group $G$.
We call the topology
on $G$ the \emph{standard topology}.
By the Birkhoff--Kakutani Theorem (see e.g. \cite{KL})
there exists a conjugation-invariant norm $\norm{\cdot}$ on $G$, inducing the standard topology.
We make the following definition, analogous (in fact, almost identical) to Definition \ref{dp}.
\begin{defin}\label{nc1}
Let $G$ be a compact metrizable group. We say that $G$ is \emph{norm complete} if every non-discrete,
conjugation-invariant norm on $G$ induces the standard topology.
\end{defin}
Norm-completeness implies the dichotomy property of $G$ as a discrete group (Definition \ref{dp}),
hence the conclusion of Proposition \ref{nsg} applies to any norm-complete group
(note that the residually finiteness was used in the proof
of Proposition \ref{nsg} only to ensure the existence of a precompact norm):
\begin{prop}\label{nsg2}
Any infinite normal subgroup of a norm complete group has finite index.
\end{prop}
Here we call a group having the property that any infinite normal subgroup of it has finite index 
a \emph{just infinite} group \footnote{This is not a standard terminology; in most of the literature \emph{just infinite} 
refers to \emph{all} non-trivial normal subgroups being of finite index.}.
As a consequence of the seminal work of Gleason--Montomery--Zippin--Yamabe towards Hilbert’s fifth problem, one observes
that any compact metrizable just-infinite group is either profinite or a semisimple Lie group (Lemma \ref{hilbert}).
Together with Proposition \ref{nsg2}, this result provides a rough classification of norm-complete groups: they are either profinite or compact real semisimple Lie groups.
In the latter case, we prove in Theorem \ref{polt15} that all compact (semi)simple groups, \emph{both real and $p$-adic analytic}, are (separable) norm-complete. 
In general however, we are quite far from having a description of the norm-complete profinite groups.
We remark that being norm-complete is equivalent to the seemingly weaker property
that any non-discrete conjugation invariant norm on $G$ is compact (see Proposition \ref{equiv}), making it
analogous to the dichotomy property (Definition \ref{dp}).
With essentially the same proof as Theorem \ref{mainthm}, using bounded elementary generation again,
we have the following result.
\begin{thm}\label{mainthm21}
Let $n$ be an integer $\ge3$. Let $A$ be a compact metrizable ring. Assume that
every non-trivial ideal $\q$ in $A$ is
of finite index in $A$. Then $\SL_n(A)$, as well as any finite-index subgroup of it, is norm complete.
\end{thm}
We prove the theorem in Section \ref{sec3}. As in Theorem \ref{mainthm}, an important
component of the proof is the fact that $\SL_n(A)$ has bounded elementary generation.
In fact, for compact rings bounded elementary generation occurs in rank $1$ as well, suggesting
that $\SL_2(A)$, as well as its finite-index subgroups, is norm-complete for any compact metrizable ring $A$ as in the theorem.
Note that various cases like that of $\SL_2(\Z_p)$ follow from the aforementioned result (Theorem \ref{polt15}), covering
all simple $p$-adic analytic profinite groups.

Given a topological metrizable group $G$, it is natural to ask when its topology is induced by 
a conjugation-invariant norm (i.e. a bi-invariant metric). Clearly, a necessary condition is the existence of a system 
of neighborhoods at the identity, each of which is invariant under all conjugations. 
As noted by Klee, it follows from the work of 
Birkhoff--Kakutani that the latter condition is also sufficient (see \cite{KL} and references therein). 
Such groups are often called in the literature SIN
(Small Invariant Neighborhoods). In \cite{Thom}, Dowerk and Thom say that a 
topological group $G$ has the \emph{invariant automatic
continuity property} (or property IAC) if every homomorphism from $G$ to any separable SIN group is continuous.

\begin{prop}\label{auto}
Any norm-complete group has property \rm{IAC}. 
\end{prop}
We prove the proposition in Section \ref{sec3}, and use it to extend Dowerk and Thom's result that 
${\rm{SU}}(n)$ has property \rm{IAC} (\cite{Thom}) to all simple connected compact groups, $p$-adic analytic simple Lie groups, 
and to the profinite groups covered in Theorem \ref{mainthm21}.
We remark that property IAC is not shared by all finitely generated profinite groups.
For instance, $\Z_p$ has many non-continuous
embeddings into $\C$, coming from abstract field isomorphisms $\C_p\cong\C$
(a similar construction can be made with the circle group).
It is interesting to ask what other profinite, or more general compact groups, have this property. 
In \cite{NIK}, Nikolov and Segal proved that every homomorphism
from a topologically finitely generated profinite group to any profinite group is continuous.
The latter property is called {\it strong completeness}, which motivated  ours {\it norm completeness}.
Proposition \ref{auto} shows that in the particular case of $G$ being a profinite norm-complete group,
the situation is even more rigid, as we do not assume the target group is profinite. 
To illustrate this phenomenon we give an example of a target group which 
is an infinite dimensional, non-locally compact $p$-adic SIN group; see the end of Section \ref{sec3}.

\begin{rem}\label{Hofer}
\footnote{This remark closely follows a comment of an anonymous referee.}
The group $D$ of Hamiltonian diffeomorphisms of a closed symplectic manifold
carries a bi-invariant Finsler metric which, starting from its discovery by Hofer in 1990 (\cite{Hof2}),
became one of the central characters in symplectic topology. The corresponding {\it Hofer topology}
is canonical in the following sense: roughly speaking, it is the unique topology associated to any smooth bi-invariant Finsler metric on $D$ inducing a non-degenerate distance function (see \cite{BO} for further details). With respect to this topology, the group $D$ is separable (see \cite{Ked},  the proof of Theorem 1.9). Therefore, any monomorphism of a separably norm complete topological group $G$ to $D$ is necessarily a homeomorphism to its image with respect to the Hofer topology. By Theorem 3.3, this is applicable to any a semi-simple compact Lie group $G$. Note that the situation changes dramatically in the smooth category: by using
a wild automorphism of $\C$, one can embed $\rm{PSU}(n)$ as a disconnected subgroup into $\rm{PSL}(n,\C)$,
which admits monomorphisms to diffeomorphism groups of manifolds (e.g., $\rm{PSL}(2,\C)$ acts on the
sphere $S^2$ by M\"{o}bius transformations) \footnote{Thanks to Kathryn Mann for this example.}.
Note however that since $\rm{PSL}(n,\C)$ is non-compact and simple, it does not admit monomorphisms
to $D$ by a result of Kedra, Libman and Martin {\cite[Theorem 1.9]{Ked}}, which is
quite coherent to our conclusion on Hofer continuity, see below.
It would be interesting to understand whether the Hofer continuity of such a monomorphism yields an obstruction to its existence. Sometimes, a regularity of a monomorphism from a Lie group to a diffeomorphism group ``upgrades itself" automatically. For instance, it is a classical result \cite{BD} that continuous (with respect to the ``usual" uniform topology) monomorphisms from a Lie group to diffeomorphism groups induce a smooth action. \footnote{By Theorem 9.10 in \cite{Kechris},
the uniform continuity would follow even from weaker regularity assumption, Baire measurability.}
If Hofer continuity can be upgraded to the uniform norm continuity, any monomorphism from $G$ to $D$ would define a smooth Hamiltonian action of $G$ on $M$, and hence numerous constraints on such would be applicable.  As an illustration, recall that a closed connected symplectic four manifold admitting an effective smooth Hamiltonian $\rm{SO}(3)$-action is diffeomorphic either to $\C \rm{P}^2$ or to the product of $S^2$ with a closed orientable surface, see the remark after Theorem 7.1 in \cite{Igl} (we thank Yael Karshon for a useful consultation.). Does there exist a monomorphism ${\rm{SO}}(3) \to D$ for $M$ being, for instance, the blow up of $\C \rm{P}^2$ at one point? If exists, such a monomorphism will be necessarily discontinuous in the uniform norm, albeit Hofer continuous. At the same time in order to rule out its existence, new ideas are needed.
\end{rem}

In the category of topological groups, beyond local compactness it is the separability property that introduces a certain "tameness" of the group.
When one insists on considering separable normed groups, some pathological constructions, such as viewing a compact group
with its discrete topology, naturally disappear. 
This leads us to consider the following weaker, yet perhaps more natural notion of norm rigidity:
\begin{defin}\label{snc1}
Let $G$ be a metrizable compact group. We say that $G$ is \emph{separable norm complete} if every separable
conjugation-invariant norm on $G$ induces the standard topology on $G$.
\end{defin}
For separable norm complete groups we prove:
\begin{thm}\label{ythm22}
Let $G$ be a separable norm complete group.
\begin{enumerate}[(i)]
\item Let $H$
be a normal subgroup of $G$,
 of countable index. Then $H$ is open (and thus of finite index). \label{part1}
\item Every finite-index subgroup of $G$ is separable norm complete. \label{part2}
\end{enumerate}
\end{thm}
We prove the theorem in Section \ref{sec3}.
We point out a relation to the work \cite{NIK3} of Nikolov and Segal on quotients of finitely generated
profinite groups.
Recall that a group $G$ is said to be FAb if every virtually-abelian quotient of $G$ is finite. When $G$
is a topological group, this refers to continuous quotients.
Nikolov and Segal showed that for a finitely generated profinite group $G$, if $G$ is FAb,
then it has no countable infinite quotients. Theorem \ref{ythm22} suggests that this result may be strengthened even further,
to the conclusion that such groups might actually be separable norm complete (see Section \ref{questions}).
\paragraph{An example of a finitely-presented meager group with possible application to the theory of approximable groups.}
Recall that a group is called \emph{meager} if any conjugation-invariant norm on it is bounded and discrete.
As an application of Theorem \ref{mainthm}, we construct an example of a finitely-presented meager group.
In \cite{Moore}, Moore constructed by number theoretic methods (following Weil’s work on the
metaplectic kernel) non-split $2:1$ central extensions $\pi: \tilde{G}\to G=\SL_n(\Q_p)$.
As an application of Theorem \ref{mainthm} we have the following result.
\begin{thm}\label{y51}
Let $\tilde{\Gamma}=\pi^{-1}(\SL_n(\Z[\frac{1}{p}]))$. Then every conjugation-invariant norm on $\tilde{\Gamma}$ is discrete
and bounded.
\end{thm}
We prove the theorem at the end of Section \ref{sec1}. It seems that all previously known meager groups were
essentially boundedly simple (from any non trivial element one can get to any other by uniformly bounded multiplication of conjugates).
This property implies uniform discreteness of all conjugation invariant norms. In the above example the abundance of (finite index)
normal subgroups makes such uniformity impossible.
We hope Theorem \ref{y51} might be of interest in the theory of approximable groups. One of the equivalent
definitions of a group $G$ being approximable with respect to a class of groups endowed with a length function
is that $G$ can be embedded in a certain large, ultra product group, endowed naturally with a bi-invariant metric.
From Theorem \ref{y51} we see that \emph{any embedding of $\tilde{\Gamma}$ in such a group must be discrete}.
This yields some inherent uniformity in all possible approximations of $\tilde{\Gamma}$, which might be useful
in trying to contradict their existence.

\paragraph{Plan of the paper.}
Let us outline the plan of the paper.
\\
\\
In Section \ref{sec1}, we prove Theorem \ref{mainthm}. After proving the first part of the theorem, we discuss the
weak dichotomy property and passage to finite-index subgroups. Then we prove the second part of Theorem \ref{mainthm}.
Finally we prove Theorem \ref{y51}.
\\
\\
In Section \ref{sec3}, we start with some generalities of our norm-rigidity notions for compact groups, showing in particular
that norm-complete groups are either semisimple Lie groups or profinite groups, and that simple Lie groups are norm-complete.
Then we prove Proposition \ref{auto} and Theorem \ref{mainthm21}.
Finally we prove Theorem \ref{ythm22}.
At the end of the section we give a construction of an infinite-dimensional separable
non-locally compact $p$-adic normed group.
\\
\\
In Section \ref{questions} we provide several questions for further research.
\\
\\
In Appendix \ref{sec42} we prove results related to conjugacy width in $\SL_n$, which are needed for
the proofs of our main results.

\paragraph{Acknowledgments.}
We would like to thank Dor Elboim, Michael Larsen, Elon Lindenstrauss, Nikolay Nikolov, Peter Sarnak and Alexander Trost for useful discussions.
We thank the anonymous referee for pointing out
an application of our results to symplectic geometry (see Remark \ref{Hofer}),
and David Fisher and Kathryn Mann for useful consultations on related topics.
We would also like to thank the anonymous referee
for the suggestion to generalize our original Theorem \ref{polt15} to its current form, and numerous useful
comments.
The third named author was supported by the ISF grant 1483/16
of the second named author during part of the work on this paper. They both thank the Israeli Science Foundation
for its support.

\section{Norm rigidity for abstract groups}\label{sec1}
We start with proving the first part of the Theorem \ref{mainthm}.
As pointed out in the introduction, proving the second part is more involved.
In particular, we will see that in general the dichotomy property
must not be preserved under passage to a finite-index subgroup (Proposition \ref{example}).
To prove the second part, we introduce a weak dichotomy property which is invariant
under passage to finite-index normal subgroups,
but turns out to be equivalent in our situation (Definition \ref{wdich}).
Finally we prove Theorem \ref{y51}.

\subsection{Proof of Theorem \ref{mainthm} \eqref{main11}}

To prove the first part of the theorem, we analyze the restrictions of our conjugation-invariant norm to certain
abelian subgroups. These subgroups are described in the following lemma.
\begin{lem}[{\bf{~\cite[Lemma 2.4]{Sh}}}]\label{structural}
Fix an integer $n\ge3$. Every elementary matrix $I+te_{ij}\in \SL_n(A)$ belongs to some copy of a subgroup isomorphic to $A^2$, contained naturally in the semi-direct product $\SL_2(A)\rtimes A^2$, itself embedded in $\SL_n(A)$.
In fact, the asserted copy $\SL_2(A)$ can always be found along the
main diagonal (i.e. occupying the entries $(i,j)$ with $i,j=k,k+1$, for some
$1\le k\le n-1$).
\end{lem}

\begin{lem}\label{lem1}
Let $A$ be an integral domain with unit, and $\q$ an infinite ideal in $A$.
Let $\norm{\cdot}$ be a non-discrete norm on $\q\oplus\q$. Assume $\norm{\cdot}$ is invariant
under the standard linear action of $\E(2,\q,A)$ on $\q\oplus\q$. Then for every $\epsilon>0$ there
exists a non-zero ideal $\q'$, contained in $\q$, such that $\q'\oplus\q'$ is contained in $D_\epsilon(0)$
(the $\norm{\cdot}$-ball of radius
$\epsilon$, centered at $0$).
\end{lem}
\begin{proof}
Let $(x,y)\in\q\oplus\q$. Then for each $z\in\q$ we have
$$
\begin{pmatrix}x\\y\end{pmatrix}-\begin{pmatrix}1&-z\\0&1\end{pmatrix}\begin{pmatrix}x\\y\end{pmatrix}
=\begin{pmatrix}0&z\\0&0\end{pmatrix}\begin{pmatrix}x\\y\end{pmatrix}=\begin{pmatrix}zy\\0\end{pmatrix},
$$
and so
\begin{equation}\label{fco}
\norm{(zy,0)}\le2\norm{(x,y)}.
\end{equation}
Similarly,
\begin{equation}\label{sco}
\norm{(0,zx)}\le2\norm{(x,y)}.
\end{equation}
Applying \eqref{fco} again we get
$$
\norm{(z'zx,0)}\le4\norm{(x,y)},
$$
for each $z,z'\in\q$.
Combining this with \eqref{sco} we find that
$$
\norm{(z'z''x,zx)}\le6\norm{(x,y)},
$$
for any $z,z',z''\in\q$.
In particular, for any $a,b\in A$
$$
\norm{(ax^3,bx^3)}\le6\norm{(x,y)}.
$$
Thus if $(x,y)\in\q\oplus\q$ satisfies $6\norm{(x,y)}\le\epsilon$, then the ideal $\q'=Ax^3$
satisfies the required property. The result now follows from the fact that $\norm{\cdot}$ is non-discrete.
\end{proof}
Notice that for this proof, we only used the fact there exists $M>0$ such that $\norm{gv}\le M\norm{v}$
for each $g\in \E(2,\q,A)$ and $v\in\q\oplus\q$.

\begin{lem}\label{lem2}
With the notation of Lemma \ref{lem1}, let $\overline{\q\oplus\q}$ be the metric completion of $\q\oplus\q$
with respect to $\norm{\cdot}$. Assume that $A$ has the property that every non-zero ideal in $A$ has
finite index. Then $\overline{\q\oplus\q}$ is profinite.
\end{lem}
\begin{proof}
Consider all subgroups of $\q\oplus\q$ of the form $\q'\oplus\q'$, where $\q'$ is a non-trivial ideal of $\q$.
Viewing them as subgroups of $\overline{\q\oplus\q}$, we may take the closure. In this way we get for each
non-trivial ideal $\q'$ of $\q$ a closed subgroup of $\overline{\q\oplus\q}$. Denote this subgroup
by $U_{\q'}$. Since $\q'$ is non-trivial it has finite index in $\q$, hence $U_{\q'}$ has finite index in
$\overline{\q\oplus\q}$. Thus $U_{\q'}$ is open.
On the other hand, by Lemma \ref{lem1} for each $\epsilon>0$ there exists a non-trivial ideal $\q'$ such that
$U_{\q'}$ is contained in the ball of radius $\epsilon$ around $0$. By taking translates we
obtain a finite cover of $\overline{\q\oplus\q}$ by open-closed cosets,
each of them is contained in a ball of radius $\epsilon$.
We conclude that $\overline{\q\oplus\q}$
is profinite.
\end{proof}
Now we are ready to complete the proof.
Let $\norm{\cdot}$ be a non-discrete conjugation-invariant norm on $\E_n(A)$, and
denote the metric completion of $\E_n(A)$ with respect to $\norm{\cdot}$ by $G$.
By Theorem \ref{mainwidth} we have that the elementary subgroups are non-discrete. Thus by Lemma
\ref{structural} and Lemma \ref{lem2}, the completion of each elementary subgroup is profinite.
Denote these completions by $U_1,\dots,U_{n(n-1)}$. Then each $U_i$ can be seen as a subgroup
of $G$, and from the fact that $\E_n(A)$ is boundedly generated
by its elementary groups we get that $G$ is boundedly generated by the $U_i$.
This means that $G=U_{i_1}\cdots U_{i_m}$ for some finite sequence $(i_1,\dots,i_m)$ ($1\le i_j\le n(n-1)$).
Hence to finish the proof
it suffices to show that for any topological group $G$, if $G$ is a finite (set-theoretic) product
$G=G_1\cdots G_m$ of a finite collection of subgroups $G_i<G$, and each $G_i$ is profinite in the relative
topology, then $G$ is profinite itself. To see this, notice first for each $i$, any finite-dimensional representation
$\pi:G_i\to \GL_d(\C)$ has finite image. This follows easily from the fact that $G_i$ is profinite and
that $\GL_d(\C)$ has no small subgroups, meaning
that there exists an identity neighborhood $U$ such that the only subgroup of $\GL_d(\C)$ contained in $U$ is the
trivial group. Hence the image of any finite-dimensional representation $\pi: G\to \GL_d(\C)$ is finite.
Since $G$ is compact, being a continuous image of the compact group $G_1\times\dots\times G_m$,
it follows that the kernel of any such representation must be open. On the other hand,
as a consequence of the Peter-Weyl theorem, each non-trivial element $g\in G$
has an open neighborhood $V_g$, lying outside of the kernel of some finite dimensional representation
$\rho_g$ of $G$. Given any open identity neighborhood $V\subset G$,
we can cover its complement by finitely many of the $V_g$ (using again the fact that $G$ is compact),
so that the kernel of
the direct sum of the corresponding $\rho_g$'s is contained in $V$. Thus for any open identity neighborhood
$V$ there exists a finite-dimensional representation $\pi$ as above,
with $\ker{\pi}\subset V$.
It follows that $G$ admits a basis
of the identity consisting of open and closed subgroups, hence $G$ is profinite.
\qed

We would like to thank Yves de Cornulier for pointing out to us the argument at the end of the proof.

\subsection{Preparations for the proof of Theorem \ref{mainthm} \eqref{main21} --
Behavior under passage to finite-index subgroups, and the weak dichotomy property}
Before we turn to the proof of the second part of the theorem, we introduce a slightly weaker version
of the dichotomy property called the \emph{weak dichotomy property}.
The following lemma states that any group having a finite-index subgroup with the (strong) dichotomy
property satisfies the (strong) dichotomy property itself.
\begin{lem}\label{up}
Let $G$ be a group, and $H$ a finite-index subgroup of $G$. Assume $H$ has the (strong) dichotomy property.
Then $G$ has the (strong) dichotomy property as well.
\end{lem}
\begin{proof}
Let $\norm{\cdot}$ a non-discrete conjugation-invariant norm on $G$. We need to show that $\norm{\cdot}$
is precompact.
To prove this, we claim that the restriction of $\norm{\cdot}$ to $H$ is non-discrete as well.
Indeed, suppose
otherwise, then $H$ is discrete (with respect to the metric on $G$ induced by $\norm{\cdot}$).
In particular, $H$ is closed in $G$. Thus all of the cosets of $H$ are closed and discrete in $G$.
Since $H$ has finite index in $G$ it follows that $G$ is discrete, which is a contradiction.
Thus the restriction of $\norm{\cdot}$ to $H$ is non-discrete, as claimed.
Since $H$ has the dichotomy property, this restriction is precompact. Since $H$ has finite index in $G$ it follows that
$\norm{\cdot}$ is precompact as well. If the metric completion of $H$ is further profinite then clearly
the same is true for $G$.
\end{proof}
It turn out that the converse statement, that the dichotomy property is preserved under passage to
a finite-index subgroup is false (Proposition \ref{example}), and this fact is our motivation
for defining the weak dichotomy property.
We start with the following simple but useful lemma.
\begin{lem}\label{ylem}
Let $G$ be a group and $N$ a normal subgroup of $G$. Let $\norm{\cdot}$ be a non-discrete conjugation-invariant
norm on $G$. Assume that the restriction of $\norm{\cdot}$ to $N$ is discrete. Then $C_G(N)$ (the centralizer
of $N$ in $G$)
contains an open neighborhood of $1$ in $G$ (with respect to the topology induced by $\norm{\cdot}$).
\end{lem}
\begin{proof}
Since the restriction of $\norm{\cdot}$ to $N$ is discrete, there exists $\epsilon>0$ such that $\norm{n}>\epsilon$
for every $1\ne n\in N$. Let $g\in G$, and suppose that $\norm{g}<\frac{\epsilon}{2}$. Then for
every $n\in N$ we have
\begin{equation}\label{eq100}
\norm{[g,n]}\le2\norm{g}<\epsilon.
\end{equation}
On the other hand $[g,n]$ belongs to $N$, since $N$ is normal in $G$. Thus
$[g,n]=1$, and so $g\in C_G(N)$. Thus the ball of radius $\frac{\epsilon}{2}$ is contained in $C_G(N)$,
finishing the proof of the lemma.
\end{proof}

\begin{cor}\label{cor12}
Let $G$ be a group, and $A,B$ subgroups of $G$. Suppose $G=AB$, a direct product. Suppose $A$ has trivial center.
Let $\norm{\cdot}$ be a non-discrete norm on $G$. Then at least one of the restrictions of $\norm{\cdot}$
to $A$ and $B$ is non-discrete.
\end{cor}
\begin{proof}
Suppose the restriction to $A$ is discrete. Then by Lemma \ref{ylem} we
have that $B$ contains a neighborhood of $1$ in $G$.
Since $\norm{\cdot}$ is non-discrete, it follows that its restriction to $B$ is non-discrete.
\end{proof}

Note that a direct product $\Gamma_1\times\Gamma_2$ of two infinite, residually finite groups cannot satisfy
the dichotomy property, due to Proposition \ref{nsg}.
Using this observation, we show that the dichotomy property must not be preserved
under passage to finite-index subgroups.
\begin{prop}[Example]\label{example}
Let $\Gamma$ be a group having the dichotomy property. Assume $\Gamma$ has trivial center. Let
$G$ be the semidirect product $G=(\Gamma\times\Gamma)\rtimes (\Z/2\Z)$,
where the action of the non-trivial element in $\Z/2\Z$ on $\Gamma\times\Gamma$
is given by switching the coordinates.
Then $G$ has the dichotomy property (while its finite-index subgroup $\Gamma\times\Gamma$
does not, as long as $\Gamma$ admits some non-discrete conjugation-invariant norm).
\end{prop}
\begin{proof}
Denote by $\Gamma_1,\Gamma_2$ the copies of $\Gamma$ in $G$ in the first and second coordinate respectively.
Let $\norm{\cdot}$
be a non-discrete, conjugation-invariant norm on $G$. We need to show that $G$ is precompact with respect
to that norm.
To prove this, note first that $\Gamma_1\Gamma_2$ is non-discrete.
Thus either $\Gamma_1$ or $\Gamma_2$ are non-discrete,
due to Corollary \ref{cor12}. Assume without loss of generality that $\Gamma_1$ is non-discrete. Then
applying the non-trivial element of $\Z/2\Z$ on $\Gamma_1$ we see that $\Gamma_2$ is non-discrete as well. Since $\Gamma$
has the dichotomy property it follows that both $\Gamma_1$
and $\Gamma_2$ are precompact. Thus $\Gamma_1\Gamma_2$ is precompact. Since $\Gamma_1\Gamma_2$
has finite index in $G$, we conclude that $G$ is precompact as well.
\end{proof}

Our example in Proposition \ref{example} uses significantly the fact that the dichotomy property is not preserved
under taking direct products. If $G=A\times B$ then our basic example of a conjugation-invariant
norm on $G$ which is neither discrete or precompact, is the singular extension of a
non-discrete norm on one of the factors $A$ or $B$ of $G$ (Construction \ref{con2}).
Denote these extension norms by $\norm{\cdot}_1$ and $\norm{\cdot}_2$ respectively. Then clearly the sum
norm $\norm{\cdot}_1+\norm{\cdot}_2$
is discrete.
As we shall see, the existence of such norms is essentially the reason that the dichotomy property
is not preserved under taking direct products. This leads us to make the following two definitions.

\begin{defin}\label{snd_def}
Let $G$ be a group. A conjugation-invariant norm $\norm{\cdot}$ on $G$ is
called \emph{strongly non-discrete} (\emph{SND}) if for any non-discrete norm $\norm{\cdot}_1$
on $G$, the sum norm $\norm{\cdot}+\norm{\cdot}_1$ is non-discrete.
\end{defin}

\begin{defin}\label{wdich}
Let $G$ be a group. We say that $G$ has the \emph{weak dichotomy property} if every SND norm on $G$
is precompact.
\end{defin}

\begin{thm}\label{ythm1}
Let $G$ be a group and let $N$ be a finite-index, normal subgroup of $G$. Assume that $G$ has the weak dichotomy property.
Then $N$ has the weak dichotomy property as well.
Similarly, if any SND norm on $G$ has a profinite completion then the same is true for $N$.
If every SND norm on $G$ is compact, then $N$ has this property as well.

\end{thm}

\begin{proof}
Let $\norm{\cdot}$ be an SND norm on $N$. Replacing $\norm{\cdot}$ by
$n\mapsto\frac{\norm{n}}{1+\norm{n}}$, we may assume that
$\norm{\cdot}$ is bounded by $1$. For each element $g\in G$ consider the function on $N$ given by
$\norm{\cdot}_g:=\norm{gng^{-1}}$. It is easy to see that $\norm{\cdot}_g$ is again an SND norm on $N$.
Finally, set $\tilde{\norm{n}}:=\frac{1}{[G:N]}\sum_{s\in G/N}\norm{n}_s$. It is easy to see that $n\mapsto\tilde{\norm{n}}$
is again an SND norm on $N$, and is bounded by $1$. Clearly we have
$\tilde{\norm{gng^{-1}}}=\tilde{\norm{n}}$ for each element $g$ in $G$.
Thus we can consider the singular extension of $\tilde{\norm{\cdot}}$ to $G$ (Construction \ref{con2}). We continue to denote
this extension by $\tilde{\norm{\cdot}}$. Since the restriction of $\tilde{\norm{\cdot}}$ to $N$ is an SND norm on $N$, we have
that $\tilde{\norm{\cdot}}$ is an SND norm on $G$. Since $G$ has the weak dichotomy property, $\tilde{\norm{\cdot}}$ is
precompact, and its restriction to $N$ is precompact. Thus $\norm{\cdot}$ is precompact. This proves the first part of the theorem.
The second and third parts follow by repeating the argument above.
\end{proof}

\begin{lem}\label{sndext}
Let $G$ be a direct product $G=N\times M$ of two groups. Assume $N$
has trivial center. Let $\norm{\cdot}_M$ be a precompact norm on $M$.
Let $\norm{\cdot}_N$ be an SND norm on $N$. For each element $g=(n,m)$ in $G$
let
$$
\norm{g}:=\norm{n}_N+\norm{m}_M.
$$
Then $\norm{\cdot}$ is an SND norm on $G$.
\end{lem}
\begin{proof}
Clearly $\norm{\cdot}$ is a non-discrete, conjugation-invariant norm on $G$. To prove that
$\norm{\cdot}$ is SND, let $\norm{\cdot}_1$ be a non-discrete norm on $G$.
Suppose $\norm{\cdot}+\norm{\cdot}_1$ were discrete. Then in particular its restriction to $N$
is discrete. Since $\norm{\cdot}_N$ is an SND norm on $N$, it follows that the restriction of
$\norm{\cdot}_1$ to $N$ is discrete as well.
Thus the restriction of $\norm{\cdot}_1$ to $M$ is non-discrete, due to Corollary \ref{cor12}. Combining this with the fact that
$\norm{\cdot}_M$ is precompact, we get that the restriction of $\norm{\cdot}+\norm{\cdot}_1$ to $M$
is non-discrete. In particular $\norm{\cdot}+\norm{\cdot}_1$ is non-discrete, which is a contradiction. This contradiction
proves the lemma.
\end{proof}

\begin{lem}\label{sndres}
Let $G$ be a direct product $G=N\times M$ of groups, and assume $N$ has trivial center.
Let $\norm{\cdot}$ be any SND norm on $G$.
Then
the restriction $\norm{\cdot}|_N$, of $\norm{\cdot}$ to $N$, is an SND norm on $N$.
\end{lem}
\begin{proof}
Let $\norm{\cdot}_N$ be
a non-discrete conjugation-invariant norm on $N$. We have to show that $\norm{\cdot}|_N+\norm{\cdot}_N$ is non-discrete as well.
For this, fix any discrete, conjugation-invariant norm $\norm{\cdot}_M$ on $M$.
Since $\norm{\cdot}_N$ is non-discrete, the sum $(n,m)\mapsto \norm{n}_N+\norm{m}_M$ defines a non-discrete,
conjugation-invariant norm on $G$. As usual, denote this norm by $\norm{\cdot}_N+\norm{\cdot}_M$.
Thus since $\norm{\cdot}$ is an SND norm on $G$, $\norm{\cdot}+\norm{\cdot}_N+\norm{\cdot}_M$ is non-discrete.
The restriction of $\norm{\cdot}+\norm{\cdot}_N+\norm{\cdot}_M$ to $N$ is
$\norm{\cdot}|_N+\norm{\cdot}_N$. If this restriction was discrete, then we would get
by Corollary \ref{cor12} that the restriction
of $\norm{\cdot}+\norm{\cdot}_N+\norm{\cdot}_M$ to $M$ is non-discrete, contradicting the fact that
$\norm{\cdot}_M$ is discrete. Thus $\norm{\cdot}|_N+\norm{\cdot}_N$ is non-discrete, finishing the proof of the lemma.
\end{proof}
\begin{cor}\label{sndcor1}
Let $G$ be a direct product $G=M\times N$, where $M,N$ are residually-finite groups with trivial centers.
Then $G$ has the weak dichotomy property if and only if both $M$ and $N$ has the weak dichotomy property.
\end{cor}
\begin{proof}
Suppose $G$ has the weak dichotomy property, and let $\norm{\cdot}_N$ be an SND norm on $N$.
We need to show $\norm{\cdot}_N$ is precompact. By Lemma
\ref{sndext} we can extend $\norm{\cdot}_N$ to an SND norm $\norm{\cdot}$ on $G$. Since $G$ has the
weak dichotomy property, $\norm{\cdot}$ is precompact. In particular $\norm{\cdot}_N$ is precompact.
A symmetric argument shows that every SND norm on $M$ is precompact. In the other direction,
suppose that $N$ and $M$ have the weak dichotomy property, and let $\norm{\cdot}$ be an SND norm on $G$.
Then the restrictions of $\norm{\cdot}$ to $N$ and $M$ are SND norms, due to Lemma \ref{sndres}. Since $N$ and
$M$ have the weak dichotomy property, these restrictions are precompact, making $\norm{\cdot}$ precompact
as well.
\end{proof}

Since the weak dichotomy property is preserved under direct products, groups having the weak dichotomy
property might fail to have the property that any infinite normal subgroup has finite-index. However, the following result
shows that this is essentially the only obstacle.
\begin{thm}\label{ythm2}
Let $G$ be a residually-finite group having the weak dichotomy property.
Assume that every finite-index subgroup of $G$ has trivial center.
Let $N$ be an infinite normal subgroup of $G$. Then there exist a finite-index, normal subgroup $G_0$ of $G$
containing $N$, and a subgroup $M$ of $G_0$, such that $G_0$ is a direct product $G_0=NM$. Moreover, $M$ can
chosen to be normal in $G$. In fact, $M$ can be taken to be $M=C_G(N)$.
\end{thm}
\begin{proof}
Let $G_0=MN$, where $M$ is the centralizer $M=C_G(N)$ of $N$ in $G$. Then both $M$ and
$G_0$ are normal in $G$. Thus it suffices to show that $G_0$ has finite index in $G$, and that $M$
and $N$ have trivial intersection. We start by showing that $G_0$ has finite index in $G$. Let
$\norm{\cdot}_{G_0}$ be a precompact norm on $G_0$, invariant under conjugations by all elements of $G$
(Construction \ref{con}). Let $\norm{\cdot}$ be the singular extension of $\norm{\cdot}_{G_0}$ to $G$ (Construction \ref{con2}).
We claim that $\norm{\cdot}$ is an SND norm on $G$. Indeed, suppose this were not so. Then there would exist
a non-discrete conjugation-invariant norm $\norm{\cdot}_1$ on $G$ such that $\norm{\cdot}+\norm{\cdot}_1$
is discrete. In particular the restriction $\norm{\cdot}_{G_0}+\norm{\cdot}_1|_{G_0}$ to $G_0$ is discrete.
Thus, since $\norm{\cdot}_{G_0}$ is precompact, $\norm{\cdot}_1|_{G_0}$ is discrete. Thus the restriction of $\norm{\cdot}_1$
to $C_G(G_0)$ is non-discrete, due to Lemma \ref{ylem}. Since $C_G(G_0)\subset C_G(N)\subset G_0$, we get
a contradiction. This contradiction proves the claim.
Since $G$ has the weak dichotomy property, it follows that $\norm{\cdot}$
is precompact. Thus, since $\norm{g}=1$ for each element $g$ in $G-G_0$,
we get that $G_0$ has finite index in $G$. Finally, $M$ and $N$ have trivial intersection since
$M\cap N$ is contained in the center of $G_0$, and this center is trivial since $G_0$ has finite index in $G$.
\end{proof}

The following lemma plays a key role in our proof
of the second part of Theorem \ref{mainthm}.
\begin{lem}\label{zlem}
Let $G$ be a group. Assume that there exists an infinite subgroup $H$ of $G$,
which is precompact
with respect to any non-discrete conjugation-invariant norm on $G$. Then the sum of any two
non-discrete conjugation-invariant norms on $G$ is again non-discrete. In particular, $G$ has the dichotomy
property if and only if $G$ has the weak dichotomy property. 
The same is true when replacing
the dichotomy property by the strong dichotomy property and the weak dichotomy
property by the property that each SND-norm has a profinite metric completion in
the previous statement.
\end{lem}
\begin{proof}
Let $\norm{\cdot}, \norm{\cdot}_1$ be non-discrete conjugation invariant norms on $G$. Then the restrictions
of $\norm{\cdot}, \norm{\cdot}_1$ to $H$ are precompact.
Thus the restriction of the sum $\norm{\cdot}+\norm{\cdot}_1$ to $H$ is precompact. In particular
this restriction is non-discrete. Thus $\norm{\cdot}+\norm{\cdot}_1$ is non-discrete as a norm on $G$.
\end{proof}

\subsection{Proof of the second part of Theorem \ref{mainthm}}
We turn to the proof of the second part of Theorem \ref{mainthm}.
Recall that in general, if $G$ is a group and $H$ is a finite-index subgroup of $G$, then $G$ has
a finite-index normal subgroup $N$ contained in $H$.
Thus Lemma \ref{up} reduces the problem to (finite-index) normal subgroups.
Let $N$ be a finite-index normal subgroup of $\E_n(A)$. Since $\E_n(A)$ has the dichotomy property
it has the weak dichotomy property, hence $N$ has the weak dichotomy property as well, due to Theorem
\ref{ythm1}. Since $N$ is normal and has finite-index in $\E_n(A)$ there exists a non-trivial ideal $\q$ of $A$
such that $\E(n,A,\q)\subset N$ (see e.g. \cite{Bass} Corollary 2.4).
We claim that the restriction to $N\cap \SL(n,A,\q)$ of any non-discrete, conjugation-invariant
norm $\norm{\cdot}$ on $N$ is non-discrete. Indeed, suppose this were not so, then
by Lemma \ref{ylem} there exists a non-trivial
sequence $(n_i)\subset C_N(N\cap\SL(n,A,\q))$, satisfying $\norm{n_i}\to0$ as $i\to\infty$.
Thus, since $\E(n,A,\q)\subset N$, $(n_i)$ is contained in the centralizer of $\E(n,A,\q)$. Since the
centralizer of $\E(n,A,\q)$ consists only of scalar matrices and there are only finitely many
such matrices in $\SL_n(A)$, we get a contradiction. This contradiction proves the claim. By applying
Theorem \ref{mainwidth} we get that the restriction to each of the elementary subgroups
$\E_{ij}(A)\cap \E(n,A,\q)$, of any non-discrete conjugation-invariant norm
on $N$ is non-discrete, hence precompact, due to Lemma \ref{lem2}. Thus $N$ has the dichotomy
property as a consequence of Lemma \ref{zlem}. This finishes the proof of Theorem \ref{mainthm}. \qed

\subsection{Proof of Theorem \ref{y51}}
We need the following lemma.
\begin{lem}\label{extn}
Let $E$ be a group and $A<E$ a finite central subgroup. Let $G=E/A$ be the quotient group.
If $G$ has the strong dichotomy property,
then so does $E$.
\end{lem}
\begin{proof}
Let $\norm{\cdot}$ be a non-discrete norm on $E$. Consider the quotient norm
$$
\norm{xA}_G:=\min_{a\in A}\{\norm{xa}\}
$$
on $G$. It is easy to see that the natural map $(E,\norm{\cdot})\to (G,\norm{\cdot}_G)$ is continuous.
If $\norm{\cdot}_G$ is discrete then clearly $\norm{\cdot}$ is discrete. We claim that if the metric completion
$\overline{G}$ (with respect to $\norm{\cdot}_G$) is profinite, then so is the metric completion $\overline{E}$
of $E$ (with respect to $\norm{\cdot}$). To see this, notice that $\overline{G}$
is isometric to $\overline{E}/A$. Thus $\overline{E}$ is profinite, due to the fact that being profinite is closed under
extensions.
\end{proof}

To prove Theorem \ref{y51}, let $\pi:\tilde{G}\to \SL_n(\Q_p)$ be any non-split finite central extension of topological groups.
By {\cite[Theorem 7.4]{SW}} the lifted group $\tilde{\Gamma}:=\pi^{-1}(\SL_n(\Z[\frac{1}{p}]))$ is not residually finite.
On the other hand, it follows from Theorem \ref{mainthm} and Lemma \ref{extn} that the metric completion of any non-discrete conjugation-invariant norm
on $\tilde{\Gamma}$ must be a profinite group.
Since profinite groups are residually finite, $\tilde{\Gamma}$ admits no residually finite completions, thus every conjugation-invariant
norm on it is discrete. It is easily seen by using the Steinberg relations that any norm on $\SL_n(\Z[\frac{1}{p}])$
must be bounded on the elementary matrices. Thus every such norm must be bounded due to the fact that $\SL_n(\Z[\frac{1}{p}])$
is boundedly generated by its elementary matrices. This clearly implies that $\tilde{\Gamma}$ is bounded as well.\qed

Notice that by the triangle inequality and finiteness of $A$, $\norm{\cdot}_G$ is bounded iff $\norm{\cdot}_E$ is, thus in fact we proved:
\begin{prop}\label{polt1}
Let $G$ be a bounded group with strong dichotomy property. Then every finite central extension of $G$,
which is not residually finite, is necessarily meager. \qed
\end{prop}

Finally, we remark that while one can show that any finite index subgroup of $\tilde {\Gamma} $ is meager as well, we do not know if this stability
property holds more generally for all meager groups.

\section{Norm rigidity for compact groups}\label{sec3}
\subsection {Generalities}
Recall that we called compact metrizable group $G$ \emph{norm-complete} if any non-discrete conjugation-invariant norm on $G$
generates the standard topology of $G$ (Definition \ref{nc1}). We say that $G$ is \emph{separable} norm complete if
any separable conjugation invariant norm on it induces the standard topology (Definition \ref{snc1}).
Throughout we shall make use of two basic properties of separable metric spaces: The first one is that 
a subspace of a separable metric space is always separable as well. The second one is that compact metric spaces 
are always separable. 
The following proposition shows that the property of being norm complete is equivalent to the seemingly weaker property that every
non-discrete conjugation-invariant norm is compact, making the notion of norm complete
analogous to the dichotomy property (Definition \ref{dp}).
\begin{prop}\label{equiv}
Let $G$ be a compact metrizble group.
Assume that every non-discrete conjugation-invariant norm
on $G$ is compact. Then any such norm generates the standard topology on $G$.
\end{prop}
\begin{proof}
Let $\norm{\cdot}_0$ be a non-discrete, conjugation-invariant norm on $G$. Then
$\norm{\cdot}_0$ is compact and we have to show that $\norm{\cdot}_0$
induces the standard topology on $G$.
Fix any conjugation-invariant norm $\norm{\cdot}$ on $G$ inducing the standard
topology, and consider the sum norm $\norm{\cdot}_1:=\norm{\cdot}+\norm{\cdot}_0$.
Let $G_0=G$ with the topology induced by $\norm{\cdot}_0$ and let $G_1$ be the
abstract group $G$ with the topology induced by $\norm{\cdot}_1$.
The identity map $\varphi:x\to x$ is a continuous bijection between $G_1$ and $G_0$.
Since $\norm{\cdot}$ and $\norm{\cdot}_0$ are compact, $\norm{\cdot}_1$ is non-discrete, hence
compact as well. Thus $G_1$ is compact.
Thus $\varphi$ is a homeomorphism. Thus
$\norm{\cdot}_1$ is equivalent to $\norm{\cdot}_0$. A symmetric argument
shows that $\norm{\cdot}_1$ is equivalent
to $\norm{\cdot}$. Thus $\norm{\cdot}_0$ is equivalent to $\norm{\cdot}$ which finishes the proof of the
proposition.
\end{proof}
We remark that although the proposition is stated for norm complete groups, the argument above proves the analogous statement for separable norm complete groups. Recall that norm-complete groups are just-infinite, meaning that every infinite normal subgroup
of a norm-complete group must be of finite index (Proposition \ref{nsg}). The following observation, which is a consequence of the seminal work
of Gleason-Montomery-Zippin towards Hilbert’s fifth problem, gives a rough classification of norm-complete groups.
\begin{lem}\label{hilbert}
Let $G$ be a compact metrizable just-infinite group. Then $G$ is either a profinite group, or
a semisimple Lie group.
\end{lem}
Since we could not find a suitable reference we give a proof here.
\begin{proof}
Let $G_0$ denote the identity connected component of $G$.
If $G_0$ is finite then it must be trivial, and in this case $G$ is profinite.
Hence, since $G$ is just infinite, we may assume $G_0$ has finite-index in $G$.
By \cite{Mont}, we have that $G$ is an inverse limit of Lie groups, $G=\lim_\alpha {G_\alpha}$.
Then for each $\alpha$, we have a surjection $\phi_\alpha:G\to G_\alpha$, and the kernel of $\phi_\alpha$
is either finite or of finite index in $G$. In the latter case, $G_\alpha$ is finite. Since $G$ has finitely many connected
components we must have at least one $\alpha$ for which the kernel of $\phi_\alpha$ is finite. But then
$G$ is a finite extension of $G_\alpha$ hence $G$ is a Lie group itself. Thus $G_0$ is a compact connected Lie group.
Let $Z_0$ denote the center of $G_0$. Since $Z_0$ is a characteristic subgroup of $G_0$, it is normal in $G$, and so $Z_0$
is either finite or of finite index. We claim that the latter possibility, that $Z_0$ has finite index, cannot occur. Indeed,
otherwise $Z_0=G_0$ is a compact torus (a product of circles). Then the torsion subgroup of $Z_0$ forms an
infinite characteristic subgroup of infinite index, hence an infinite normal subgroup of $G$ of infinite index,
and we get a contradiction. Thus $Z_0$ is finite, and so $G_0$, and hence $G$, is semisimple, completing the proof of the lemma.
\end{proof}
While we are far from being able to describe the norm-complete profinite groups, the picture is very clear for Lie groups:
\begin{thm}\label{polt15}
Let $G$ be a compact Lie group, \emph{either real or $p$-adic analytic}. If $G$ is simple, meaning its Lie algebra is simple, then $G$ is norm complete.
If $G$ is semisimple (meaning its Lie algebra is semisimple), then $G$ is separable norm complete.
\end{thm}
For the proof we need the following two results. 
The first one is the following version of Lemma \ref{extn}, whose proof we omit as in view of Proposition \ref{equiv}, 
it is essentially the same as the proof of Lemma \ref{extn}.
\begin{lem}\label{extn2}
Let $E$ be a compact metrizable group and $A<E$ a finite central subgroup. Let $G=E/A$ be the quotient group.
If $G$ is norm-complete, then so is $E$. If $G$ is separable norm-complete, then so is $E$. \qed
\end{lem}
The second result we need is the following local surjectivity result. 
\begin{lem}\label{ift}
Let $k$ be either $\R$ or $\Q_p$. 
Let $M$ be a $k$-analytic manifold of dimension $n$, and fix a metric on $M$ compatible with the topology. 
Fix a point $e\in M$ (the ``origin'').  
Let $F:M\times k^n\to M$ be an analytic function.
Assume that there exists an open subset $\Omega\subset M$
such that for every $x\in\Omega$ we have:
\begin{enumerate}
\item $F(x,0)=e$.

\item $F(x,\cdot)$ is non-singular at $0$; that is the differential 
$$D_0F(x,\cdot):k^n\to T_eM$$ 
is invertible. 

\end{enumerate}
Then, for every open neighborhood $U\subset k^n$ of $0$, if we denote by $r_x$ ($x\in\Omega$) 
the maximal radius of a ball around $e$ contained in $F(x,U)$, 
then $r_x>0$, and is locally bounded below. 
\end{lem}

\begin{proof}
The result for $k=\R$ is standard and follows immediately from an appropriate version of the inverse function theorem.
The case $k=\Q_p$ is proved similarly but for the sake of completeness we give a proof.
We shall use standard notation and facts regarding $p$-adic analytic functions (see e.g. \cite{SerreBook} Part 2, Chapter 2).  
By taking coordinate charts it obviously suffices to assume $M=\Q_p^n$ and $e=0$.
Let $x_0\in \Omega$ and $V\subset \Omega$ a small compact neighborhood of $x_0$.
Let $x\in V$. 
The function $x\mapsto D_0F(x,\cdot)$ is analytic, so by multiplying by $D_0F(x,\cdot)^{-1}$, 
we may without loss of generality assume that the coordinate function $F_i(x,t)$, viewing it as an analytic function in $t$,
has the form 
$$F_i(x,t)=t_i+\sum_{\abs{\alpha}>1} a_{i,\alpha,x}t^\alpha,$$
where $a_{i,\alpha,x}\in \Q_p$. Here by convention $\alpha$ is an $n$-tuple $\alpha=(\alpha_1,\dots,\alpha_n)$, $\alpha_i\ge0$,$\in\Z$, $t^\alpha=t_1^{\alpha_1}\cdots t_n^{\alpha_n}$, and $\abs{\alpha}=\sum_i\alpha_i$. 
Since $F$ is analytic we may assume, by replacing the function $F$ by $(x,t)\mapsto \mu^{-1}F(x,\mu t)$, 
where $\abs{\mu}_p$ is sufficiently small, that $a_{i,\alpha,x}\in\Z_p$. 
Thus without loss of generality $F(x,\cdot)$ is an analytic function whose differential at $t=0$ is the identity, and whose
coefficients in the local expansion at $t=0$ are $p$-adic integers. 
The lemma now follows from the fact that such a function has an inverse on the open polydisk 
$P_0(1)=\{t \in\Q_{p^n}\mid \abs{t_1}_p,\dots,\abs{t_n}_p<1 \}$ of radius $1$ around $0$.
For the convenience of the reader we sketch the proof of the last fact (following the proof of the Inverse Function Theorem, \cite{SerreBook}, Part 2, Chapter 2): 
One can invert $f:=F(x,\cdot)$ as follows. First, 
write $\phi_i(t)=t_i-f_i(t)$ and find a (unique) formal solution to the equations 
$$
T_i=\psi_i(T)-\phi_i(\psi(T))\quad (1\le i\le n),
$$
of the form $\psi_i=\sum_{\beta>0}b_{i,\beta}T^\beta$. Observe by induction that $b_{i,\beta}=p_{\beta}^i(a_{j,\alpha,x})$ where 
$p_{\beta}^i$ is a polynomial with positive integral coefficients independent of the $\{\phi_i\}$, depending only on the 
$a_{j,\alpha,x}$'s for $|\alpha|<|\beta|$. Thus $b_{i,\beta}\in\Z_p$, so that the $\{\psi_i\}$ 
converge in $P_0(1)$ as claimed.
\end{proof}

\begin{proof}[Proof of Theorem \ref{polt15}]
We start with the case where $G$ is simple. 
Since $G$ has finite center, we may assume without loss of generality that $G$ has no center, due to Lemma \ref{extn2}.
Assuming this, let $d:G\times G\to \R$ be any bi-invariant metric on $G$ compatible with the topology on $G$. For each $x\in G$,
let $M(x)$ be the set of all conjugates and inverses of elements of the form $[x,y]$ for $y\in G$.
That is, 
$$
M(x)=\{g[x,y]^j g^{-1}\mid g,y\in G, j=\pm1\}.
$$
Let $n$ denote the dimension of $G$. By {~\cite[Lemma 3.4]{Ked}}, for any non-trivial
$x\in G$, $M(x)^{2n}$ contains an identity neighborhood. Note that in \cite{Ked}, this lemma is stated and proved in the case of a real 
group, however the exact same proof works verbatim in the case of $p$-adic analytic groups. 
For each $1\ne x\in G$, let $r_x>0$ be the maximal radius for which $B_{r_x}$, the $d$-ball of radius $r_x$ centered at $1$,
is contained in $M(x)^{2n}$. We claim that $r_x$ is locally bounded below. To see this we repeat 
the argument in the proof of {~\cite[Lemma 3.4]{Ked}} and then apply Lemma \ref{ift}: Let $x_0\in G$ be any non-trivial element. Let $Y\in\g$
be any element such that $X_x:={\rm{Ad}}(x)Y-Y\ne0$, for any $x$ in a small enough neighborhood of $x_0$.
Using the simplicity of $\g$ we get that there exist elements $g_1,\dots,g_n\in G$ so that 
$\{{\rm{Ad}}(g_i)X_{x_0}\}_{i=1}^n$, is a basis of $\g$, and so the same is true when replacing $x_0$ by any 
$x$ close enough to $x_0$.
Consider the function $F(x,\cdot):k^n \to G$, defined by 
$$
F(x,t)=\prod_{i=1}^ng_i[x,\exp(t_iY)]g_i^{-1}.
$$
Here $k=\R$ in the real case and $\Q_p$ in the $p$-adic case. 
A direct computation now shows that $D_0F(x,\cdot)$ is invertible (sending the partial derivative $\partial_i$ to ${\rm{Ad}}(g_i)X_x$).
It is clear that $F$ is analytic, and $F(x,0)=1$, and so $F$ satisfies the conditions of Lemma \ref{ift}, proving that $r_x$ is indeed locally 
bounded below. 
Now, let $\norm{\cdot}$ be a non-discrete conjugation-invariant norm on $G$. We claim that there exists
a sequence $g_k$, converging to $1$ with respect to both $\norm{\cdot}$, and $d$, which is not eventually constant. 
This follows immediately from the straightforward fact that the sum of any compact and non-discrete norms is non-discrete.  
To finish the proof, we shall use such a sequence $g_k$ to show
that $\norm{\cdot}$ is continuous at $1$. For this, let $x_i$ be any sequence converging to $1$ in the standard topology.
Then for every $k\in\N$ there exists $i_k\in \N$ such that for every $i\ge i_k$, $x_i\in M(g_k)^{2n}$. Thus
for all $i\ge i_k$, $\norm{x_i}\le 4n\norm{g_k}$. Since $\norm{g_k}\to0$, we get $\norm{x_i}\to0$, proving
that $\norm{\cdot}$ is continuous at $1$. Equivalently, the identity map $\rm{id}:(G,d)\to (G,\norm{\cdot})$
is continuous at $1$. It follows that $\rm{id}$ is continuous at every point, and since $G$ is compact it is
a homeomorphism. Hence $\norm{\cdot}$ induces the standard topology on $G$, finishing the proof of the first part of the theorem.

For the second part, in the real case recall that a compact semisimple Lie group is finite central extension of a direct sum of finitely many
simple Lie groups. 
Clearly a direct sum of finitely many separable norm-complete groups is again separable norm complete (using the fact that 
a subspace of any separable metric space is again separable), so that in the real case the second part of the theorem follows from Lemma \ref{extn2}.
For the $p$-adic case, we shall make use of the general fact that $G$ has a finite-index subgroup $G'$ which is a direct sum of simple Lie groups. 
To see this, let $\g=\g_1+\dots+ \g_s$ (direct sum) denote the Lie algebra of $G$ where $\g_i$ are simple Lie algebras (over $\Q_p$). 
Let $U=U_1+\dots+ U_s$ ($U_i\subset \g_i$) be a compact open neighborhood of $0$ in $\g$ such that the exponential map is well-defined on $U$, and is a homeomorphism onto an open subgroup of $G$ (see {\cite[Chapter 2.7.2, Proposition 3]{Bour}} for the existence of such a neighborhood $U$). For each $i$ let $G_i$ denote the subgroup of $G$ generated by $\exp(U_i)$. 
Then $G_i$ is a closed subgroup of $G$ corresponding to the subalgebra $\g_i$, and so the $G_i$'s commute
with each other (pointwise). Thus for any distinct $i,j$ the intersection $G_i\cap G_j$ lie in the center of both $G_i$ and $G_j$. On the other hand, for each $i$, since $\g_i$ is simple and $G_i$ is compact, the center of $G_i$ is finite. Thus by using the fact that the $G_i$'s are profinite, hence residually finite, we can replace each $G_i$ by a finite index subgroup $G_i'$ in a way that $G_i'$ and $G_j'$ intersect trivially for any distinct $i,j$. 
The subgroup $G'=G_1'\cdots G_s'$ has finite-index in $G$ and is a direct sum of simple Lie groups, finishing the proof of the general fact stated above. 
Using this fact, the result follows from Lemma \ref{up}. 
\end{proof}

We remark that semisimple Lie group are not, in general, norm-complete, due to Proposition \ref{nsg}. 

\subsection{Proof of Proposition \ref{auto}}
Let $f:G\to H$ a homomorphism from a norm complete group $G$ to a separable SIN group $H$. 
We assume that $G$ is not a finite group as otherwise the claim is trivial. Let $N$ denote the kernel of $f$.
Then we have by Proposition \ref{nsg2} that $N$ is either finite or of finite index. If $N$
is of finite index then it is open in $G$, due to Theorem \ref{ythm22}, and the statement is clear.
Hence we assume that $N$ is finite.
Let $G_0$ be the connected component of $G$. Then $G_0$ is a normal subgroup of $G$ and so
$G_0$ is either finite, in which case it must be trivial, or of finite index in $G$, due to Proposition \ref{nsg2}.
Assume first that $G_0$ is trivial. Then $G$ is profinite, and so there exists an open subgroup $G'$ of $G$ which
intersects $N$ trivially.
Let $\norm{\cdot}$ be a conjugation-invariant norm on $H$, generating the standard topology on $H$.
Then the function $g\mapsto \norm{f(g)}$ induces a conjugation-invariant norm on $G'$. Denote this norm
by $\norm{\cdot}_{G'}$. Since $G'$ is uncountable (as an infinite compact group), its image in $H$ is non-discrete,
due to the fact that $H$ is separable, making $\norm{\cdot}_{G'}$ non-discrete. Thus $\norm{\cdot}_{G'}$ induces
the standard topology on $G'$, making the restriction of $f$ to $G'$ continuous. Since $G'$ is open in $G$ we have
that $f$ is continuous as well.	
Next suppose $G_0$ has finite index in $G$.
Put $L=G/N$ and let $L_0$ denote the identity connected component of $L$.
Then $L_0$ is of finite index in $L$, and moreover,
$L$ is just-infinite, since $G$ is.
Since $L_0$ has finite index in $L$, $L$ cannot be profinite (otherwise $L_0$ would be trivial,
making $L$ a finite group, contradicting the fact that $N$ has infinite index in $G$), 
hence $L$ is a semisimple Lie group, due to Lemma \ref{hilbert}.
Thus by Theorem \ref{polt15} we have that $L$ is separable norm complete.
Consider the induced injection $\bar{f}:L\to H$. The function $l\mapsto \norm{\bar{f}(l)}$
is a conjugation invariant norm on $L$, which is clearly separable, due to the separability of $H$. Thus this norm induces
the standard topology on $L$, making $\bar{f}$, and hence $f$, continuous. \qed

There is a rich literature around the subject of automatic continuity of homomorphisms in various settings –
see for example the list at the beginning of \cite[Section 8]{Thom}, and in particular the survey \cite{Ros}
where some negative results for compact Lie groups are illustrated as well.

\subsection {Proof of Theorem \ref{mainthm21}}
We need the following lemma, analogous to Lemma \ref{lem2}.
\begin{lem}\label{lem3}
Let $A$ be a compact metric ring, and assume every non-trivial ideal in $A$ has finite index. Let $\q$ be a non-trivial
ideal in $A$. Let $\norm{\cdot}$ be a non-discrete norm on $\q\oplus\q$, and assume $\norm{\cdot}$ is
invariant under the standard action of $\E(2,A,\q)$ on $\q\oplus\q$. Then $\norm{\cdot}$ is profinite.
\end{lem}
\begin{proof}
From Lemma \ref{lem1} we have
that for every $\epsilon>0$ there is a non-zero ideal $\q'\subset\q$ such that
$\q'\oplus \q'\subset D_\epsilon(0)$, where $D_\epsilon(0)$ is the $\norm{\cdot}$-open ball of radius $\epsilon$ centered
at $0$. On the other hand, any such $\q'$ is closed in $\q$ (since it contains a non-trivial principal ideal $aA$ which is automatically closed, and of finite index by our assumption on $A$) and of finite index, and so it is open in $\q$ with respect to the
standard topology on $\q$ (the relative topology, induced from the standard topology on $A$). By taking translates we obtain
that every $\norm{\cdot}$-open subset of $\q\oplus\q$ is open with respect to the standard topology as well.
Since the standard topology is compact, it follows that $\norm{\cdot}$ is compact as well. The sets $\q'\oplus\q'$ with $\q'$
as above form a basis of $0$ for the topology induced by $\norm{\cdot}$, hence $\norm{\cdot}$ is profinite.
\end{proof}

\begin{proof}[Proof of Theorem \ref{mainthm21}]
Recall that for any compact ring $A$, $\SL_n(A)$ is boundedly generated by its elementary matrices (\cite{Sh}).
Thus the result follows by repeating verbatim the proof of Theorem \ref{mainthm}, using Lemma \ref{lem3} instead of
Lemma \ref{lem2}.
\end{proof}

\subsection {Proof of Theorem \ref{ythm22}}
Fix a norm $\norm{\cdot}$ on $G$ which induces the standard topology on $G$.
Then $\norm{\cdot}$ is compact, and in particular is separable. Thus $\norm{\cdot}|_H$ is
separable, due to the fact that any subspace of a separable space is separable.
Let $\delta$ be the characteristic function of the complement $G-H$.
It is obvious that $\norm{\cdot}_1:=\norm{\cdot}+\delta$ is a conjugation-invariant
norm on $G$, and we claim that $\norm{\cdot}_1$ is separable. To see this, let $C$ be a countable dense
subset of $H$ (with respect to the topology induced by $\norm{\cdot}|_H$), and let $C'$ be the
union $C'=\bigcup_{s\in S}sC$, where $S$ is some countable set of representatives for $G/H$.
Then $C'$ is countable, and it is easy to see that $C'$ is dense in $G$, with respect to the topology
induced by $\norm{\cdot}_1$. Hence $\norm{\cdot}_1$ is separable, as claimed. Since $G$ is separable
norm complete it follows that $\norm{\cdot}_1$
is compact. Suppose now that $H$ were not open. Then there would exist a sequence $g_i$ of elements
in the complement of $H$ with $\norm{g_i}\to 0$. Since $\norm{\cdot}_1$ is
compact, we may assume that $\norm{g_ig}_1\to0$ for some element $g\in G$.
Since $\norm{\cdot}\le \norm{\cdot}_1$ we get that $\norm{g_ig}\to 0$, hence $g=1_G$.
On the other hand $\norm{g_i}_1\ge1$, which is a contradiction. This contradiction proves
\eqref{part1}.

To prove \eqref{part2}, let $H$ be a subgroup of $G$ of finite index, and let
$\norm{\cdot}$ be a separable norm on $H$. We have to show that $\norm{\cdot}$ is compact.
We may assume that $\norm{\cdot}\le 1$ (see e.g. in the beginning of the proof of Theorem \ref{ythm1}).
Using the fact the any subspace of a separable metric space is again separable it is easy to modify the 
proof of Lemma \ref{up} to show that any metric compact group is separable norm complete whenever 
it has a separable norm complete, finite index subgroup. Thus 
by passing to a finite-index subgroup, we may assume $H$ is normal in $G$ (see the proof of Lemma \ref{up}).
For each element $g$ in $G$ let $\norm{\cdot}_g$ be the conjugation-invariant norm on $H$
defined by $x\mapsto \norm{gxg^{-1}}$, and set
$\norm{\cdot}_1:=\sum_{s\in S}\norm{\cdot}_s$, where $S$ is some set of representatives for $G/H$.
It is easy to see that $\norm{\cdot}_1$ is a conjugation-invariant norm on $H$, invariant under conjugations
by all elements of $G$.
We keep denoting by $\norm{\cdot}_1$ the singular extension of $\norm{\cdot}_1$
to $G$ (see Construction \ref{con2}). As in the proof of \eqref{part1}, it is easy to see that $\norm{\cdot}_1$ is
separable. Since $G$ is separable norm complete, it follows that $\norm{\cdot}_1$ induces the standard topology on $G$.
Hence, using $\eqref{part1}$, $H$ is open in $G$ with respect
to $\norm{\cdot}_1$, and so the restriction $\norm{\cdot}_1|_H$ is compact.
Since $\norm{\cdot}\le \norm{\cdot}_1$, we get that $\norm{\cdot}$ is compact as well, which finishes the proof of the theorem.\qed

\subsection{A construction of an infinite-dimensional, separable, non-locally compact, $p$-adic group}\label{infinite}
As indicated in the introduction, Proposition \ref{auto} holds for a quite large family of target groups
$H$. To demonstrate it, we construct an infinite dimensional $p$-adic group which is not locally compact, yet separable.
For this, let $R_p$ be the valuation ring of $\C_p$. We let $\GL(R_p):=\cup_n\GL_n(R_p)$. The natural
topology on $\GL(R_p)$, coming from the $p$-adic absolute value, is induced from a conjugation-invariant norm.
Indeed, on each $\GL_n(R_p)$
we can define a conjugation-invariant norm by letting
$\norm{g}_n:=\norm{I_n-g}_{op}$, where $\norm{\cdot}_{op}$ is the
usual operator norm:
$$
\norm{X}_{op}:=\sup_{\norm{v}_\infty=1}(\norm{Xv}_\infty).
$$
Here $\norm{\sum_iv_ie_i}_\infty=\sup_i|v_i|_p$, where $e_i$ is the standard basis of $R_p^n$ and
$|\cdot|_p$ is the $p$-adic absolute value.
It can be verified that $\norm{\cdot}_n$ generates the standard topology on $\GL_n(R_p)$.
To define a norm on $\GL(R_p)$ we simply take the limit of the norms $\norm{\cdot}_n$. Denote the resulting norm by
$\norm{\cdot}$, and let $H$ be the metric completion of $\GL(R_p)$ with respect to this norm.
Then $H$ can be described as follows: As an abstract group $H$ is the collection of all infinite
matrices of the form $I+(a_{ij})_{i,j\in\N}$ with $a_{ij}\in R_p$ and $a_{ij}\to0$ as $i+j\to\infty$, which have
inverse of the same kind. The topology on $H$ is the topology induced by the supremum-norm.
As a consequence of Proposition \ref{auto} we obtain that every abstract homomorphism $\varphi:\SL_n(\Z_p)\to H$
is already continuous.

\section{Questions}\label{questions}
We close the paper by posing several questions that naturally arise from our work.

\paragraph{1. The dichotomy property for higher-rank lattices.}
Let $\Gamma$ be a so-called “higher-rank” arithmetic lattice, or, more generally $S$-arithmetic lattice of a simply connected simple algebraic group $\G$ defined over a global field $K$.
In our paper we showed that when
$\G=\SL_{n \ge 3}$, $\Gamma$ has the dichotomy property,
but it is very natural to conjecture that, as with most elements of the rigidity theory of arithmetic groups,
this should hold for  all $\G$ and $K$ (and $S$). The most non-trivial ingredient in our approach is the use of bounded generation, hence we certainly expect (although it would still be nice to establish) that as first step these results can be extended to other Chevalley groups like $\Sp_{2n}$.

Two especially challenging cases seem to stand out: The first is that when $\G$ is anisotropic over $K$, i.e. $\Gamma$ has no unipotent elements
(and is a co-compact lattice in its “natural envelope”; see \cite[Section 6.8]{DW2} for more detalis, as well as
a construction of such groups by means of maximal orders in division algebras).
Recall from the introduction that the dichotomy property implies being just infinite,
and in these cases, even that of $\G(\R)=\SL_3(\R)$, there is no known algebraic proof of the latter (in fact no other proof besides that of Margulis, invoking
property (T) and amenability).

In the original version submitted for publication we discussed the question
of $\G=\SL_2$ with $K=\Q$ but $S$ containing some finite prime, i.e., $\Gamma = \SL_2(\Z[1/p])$.
Here there is an algebraic proof for the normal subgroup theorem, but a main ingredient in our approach breaks completely:
Unlike in the proof of Theorem \ref{mainthm}, there does exist a norm on the unipotent subgroup
$\Z[1/p]$ invariant under its normalizer in $\Gamma$ (acting on it via multiplication by powers of $p$), which is neither discrete nor precompact.
A construction of such a norm was given to us by Dor Elboim, to whom we are thankful. We refer the reader to the previous version of the paper, appearing on arXiv, for a sketch 
of Elbom's construction. Fortunately, following the distribution of that version, A. Trost \cite{Trost3} found a way around this issue, by adapting Lemma \ref{se4}.

\paragraph{2. Norm-completeness for finitely-generated FAb profinite groups.}
In their deep work on finitely generated profinite groups (\cite{NIK, NIK3}),
Nikolov and Segal have shown that such groups $G$ which are so-called FAb
have the property that their normal subgroups of countable index must in
fact be of finite index (and open) \cite[Corollary 5.24]{NIK3}. If the group $G$ is furthermore just infinite in the topological sense, then they show that it must be so in an abstract sense as well -- any infinite normal subgroup must have finite index \cite[Corollary 1.5]{NIK3}.
Recall that in our paper we established these two same properties: The first for $G$ which is separable norm-complete, and the second for $G$ which is norm-complete, and it seems very natural to wonder if those two norm rigidity properties are the “reason” underlying these deep results of Nikolov--Segal, i.e., \emph{Is it true that a finitely generated profinite FAb group must be separable norm complete, and one which is furthermore just-infinite must be norm complete?}

\paragraph{3. Norms on $\Z^2$ invariant under thin subgroups of $\SL_2(\Z)$.}

Let $\norm{\cdot}$ be a non-discrete norm on $\Z^2$ (viewed as an additive
group $\Z^2=(\Z^2,+)$), and assume $\norm{\cdot}$ is invariant under the usual linear action of
some subgroup $\Gamma<\SL_2(\Z)$.
If $\Gamma=\left<\begin{pmatrix}1&m\\0&1\end{pmatrix},\begin{pmatrix}1&0\\m&1\end{pmatrix}\right>$ for some $m\in\N$,
then $\norm{\cdot}$ must be precompact, due to Lemma \ref{lem2}.
It is well known that if $m\ge3$ then $\Gamma$ has infinite index
in $\SL_2(\Z)$, making it a \emph{thin} subgroup, i.e. an infinite co-volume, discrete, Zariski dense subgroup of $\SL_2(\R)$.
An interesting question is whether any thin subgroup has this property. That is,

\emph{Must any non-discrete norm on $\Z^2$ invariant under the linear action of a thin subgroup $\Gamma<\SL_2(\Z)$
be pre-compact?}

\noindent Note that there do exist non-discrete, unbounded norms on $\Z^2$ which are invariant under the action of a cyclic
subgroup of $\SL_2(\Z)$. For instance, let $\norm{\cdot}$ be a norm on $\Z^2$ defined by
$$
\norm{\cdot}: (x,y)\to \max\{\frac{|x|_p}{2},|y|\},
$$
where $|\cdot|_p$ is the $p$-adic absolute value and $|\cdot|$ is the usual absolute value on $\Z$.
It is easy to check that $\norm{\cdot}$ is a non-discrete, unbounded norm on $\Z^2$, invariant under the action
of the matrix $\begin{pmatrix}1&1\\0&1\end{pmatrix}$. It seems likely possible to construct a non-discrete, unbounded norm
invariant under the action of a hyperbolic matrix as well.

A possible approach to attack the aforementioned question is as follows: For a subgroup $\Gamma<\SL_2(\Z)$
and a positive number $M>0$, let $S_M(\Gamma)$ denote the set of sums:
$$
S_M(\Gamma):=\{\sum_{i=1}^l\gamma_i\mid \gamma_i\in \Gamma, l\le M\}.
$$
Suppose that $S_M (\Gamma)$ contains some congruence subgroup of $SL_2(\Z)$ (and in particular a subgroup of the form
$H=\left<\begin{pmatrix}1&m\\0&1\end{pmatrix},\begin{pmatrix}1&0\\m&1\end{pmatrix}\right>$ for some $m$).
If $\norm{\cdot}$ is any non-discrete norm on $\Z^2$ invariant under the action of $\Gamma$, then for $v\in\Z^2$ and $h\in H$
we have
$$
\norm{hv}\le M\norm{v}.
$$
It follows from the remark after the proof of Lemma \ref{lem1} and Lemma \ref{lem2} that $\norm{\cdot}$
is pre-compact (in fact, its metric completion is profinite). This raises the following question, which
seems of interest in its own right, being related to various "sum-product" phenomena in the deep theory of so-called "super-strong approximation":

\emph{Let $\Gamma<\SL_2(\Z)$ be a thin subgroup. Does there exist a positive number $M>0$
such that the set $S_M(\Gamma)$ defined above contains some congruence subgroup of $SL_2(\Z)$?}

(in fact, $S_M(\Gamma)$ should contain the group of all scalar congruence matrices, which is invariant under addition.)

The above discussion shows that any thin subgroup $\Gamma < SL_2(\Z)$ for which the answer is positive has the property that
any non-discrete norm it preserves on $\Z ^2$ must be pre-compact.

To see a non-trivial example where one has positive answer to this question note that for any positive integer $m$, any element of the full congruence subgroup $\SL(2,\Z,m)=\ker(\SL_2(\Z)\to \SL_2(\Z/m^2\Z))$ can be written as a sum of a bounded number (depending on $m$) of elements in
the thin subgroup $\Gamma = \left<\begin{pmatrix}1&m\\0&1\end{pmatrix},\begin{pmatrix}1&0\\m&1\end{pmatrix}\right>$, as the following computation shows:

\begin{align*}
\begin{pmatrix}1+m^2a&mb\\mc&1+m^2d\end{pmatrix}&=
(2m^2-3)\begin{pmatrix}1&0\\0&1\end{pmatrix}
+\begin{pmatrix}1&m(b-2)\\0&1\end{pmatrix}\\
&+\begin{pmatrix}1&0\\m(c-a-d+4)&1\end{pmatrix}
+\begin{pmatrix}1+m^2(a-2)&m\\m(a-2)&1\end{pmatrix}\\
&+\begin{pmatrix}1&m\\m(d-2)&1+m^2(d-2)\end{pmatrix}.
\end{align*}
Of course, one can naturally extend the problem above for all $n\ge2$, and more general rings $A$ in place of $\Z$.

\appendix\section{Production of elementary matrices}\label{sec42}
The results of this appendix are related to \emph{conjugacy width} in $\SL_n$.
We feel that such results should be standard, but as we were unable to find suitable versions
in the literature we give proofs. Our proof of Theorem \ref{mainwidth}
relies on computations from \cite{Bass}.
Let $A$ be an integral domain with unit, and fix an integer $n\ge2$.
We fix notation for some subgroups of $\SL_n(A)$.
For each pair of integers $1\le i\ne j\le n$ we denote by $e_{ij}$ the $n\times n$ matrix
with $1$ in the $(i,j)$-entry and $0$ elsewhere. We denote by $I$ (or $I_n$ if we want to emphasize
the dimention) the identity matrix in $\SL_n(A)$.
We denote by $\E_{ij}(A)$ the group of matrices of the form $I+ae_{ij}$, where $a\in A$.
A matrix of this form is called an \emph{elementary matrix}.
We denote by $\E_n(A)$ the subgroup of $\SL_n(A)$ generated by all elementary matrices.
For any non-zero ideal $\q$ in $A$,
we denote by $\E(n,A,\q)$ the subgroup of $\E_n(A)$
generated by the elementary matrices with off-diagonal entries belonging to $\q$.
We denote $\SL(n,A,\q):=\text{ker}(\SL_n(A)\to\SL_n(A/\q))$ (the map being the natural projection).
If it is clear from the context what $A$ and $n$ are, we denote also $\SL(n,A,\q)=\Gamma(\q)$.

Let $G$ be any group and $g$ an element $g\in G$. Let $S$ be any subset of $G$.
We define
\begin{equation}
g^S:=\{sg^is^{-1}\mid s\in S,i=\pm1\}.
\end{equation}
Consider the group $\left<g^S\right>$ generated by $g^S$.
For any subgroup $H$ of $G$ we can ask what is the minimal length of a non-trivial element
in $\left<g^S\right>\cap H$, with respect to the set of generators $g^S$.
The main result of this section is obtaining a uniform bound on this quantity, over all non-central
elements in $G$, where
$G=\SL(n,A,\q)$, $S=\E(n,A,\q)$, and $H$ an elementary group $H=\E_{ij}(\q)$.
Here $n$ is any integer $n\ge3$, $A$ is an integral domain satisfying $\SR_{n-1}$
(see Definition \ref{srdef} below), and $\q$ is
any non-zero ideal in $A$. In fact, the bound we obtain is \emph{absolute} ($2^9$), although
we do not make any use of this fact.

We start by recalling the definition of the \emph{stable range} of a ring.
\begin{defin}\label{srdef}
Let $A$ be a commutative ring with unit.
For $r\in\N$, an element $v=\sum_{i=1}^rv_ie_i\in A^r$ is called \emph{unimodular} if
$\sum v_iA=A$. Here $e_i$ ($1\le i\le r$) is the standard basis of $A^r$ ($1$ is the $i$th coordinate
and $0$ elsewhere).
$A$ satisfies $\SR_m$ if for any unimodular element $v=\sum_{i=1}^{m+1}v_ie_i\in A^{m+1}$
there exist elements $t_1,\dots,t_m\in A$ such that $\sum_{i=1}^{m}(v_i+t_iv_{m+1})e_i\in A^m$ is
unimodular.
The smallest integer $m$ such that $A$ satisfies $\SR_m$ is called the \emph{stable range of $A$}.
\end{defin}

\begin{thm}\label{mainwidth}
Let $n\ge3$. Let $A$ be an integral domain with unit, satisfying $\SR_{n-1}$. Let $\q$ be a non-zero ideal in $A$.
Let $\sigma$ be any non-central element in $\SL(n,A,\q)$. Let $1\le i\ne j\le n$. Then
there exists a non-trivial element in $\left<\sigma^{\E(n,A,\q)}\right>\cap\E_{ij}(\q)$ of length at most $2^9$ with
respect to the set of generators $\sigma^{\E(n,A,\q)}$.
\end{thm}

For the rest of the section, we fix $n\ge3$, $A$ a ring satisfying $\SR_{n-1}$, $G=\SL(n,A,\q)$,
and $S=\E(n,A,\q)$.
We call a transformation $T:G\to G$ a \emph{$\q$-operation} if $T$ is either of the form
$g\mapsto sg s^{-1}$, $g\mapsto [g,s]$, or $g \mapsto [s,g]$, where $s\in S$.
To prove Theorem \ref{mainwidth}, we will show that any non-central element $\sigma\in G$
can be reduced to a non-trivial element of $\E_{ij}(\q)$, using at most $9$ $\q$-operations.
We shall make use of the following facts, which are easily verified:
\begin{enumerate}[(i)]
\item\label{fact1} Let $g,h$ be any two matrices in $\GL_n(A)$. Let $1\le i\ne j\le n$.
Let $\alpha$ be the $i$-th column of $g$ and $\beta$ the $j$-th column of $h$.
Then
\begin{equation}
ge_{ij}h=\alpha\beta.
\end{equation}

\item\label{fact2} Let $x,y\in \GL_{n-1}(A)$, $v\in A^{n-1}$, and $u\in A^\times$.
Then
\begin{equation}
\left[\begin{pmatrix}u&v\\0&x\end{pmatrix},\begin{pmatrix}1&0\\0&y\end{pmatrix}\right]=
\begin{pmatrix}1&(vy-v)(yx)^{-1}\\0&[x,y]\end{pmatrix}.
\end{equation}
Here we view $v$ as a row vector.
\item\label{fact3}  In the notation of \eqref{fact2},
\begin{equation}
\left[\begin{pmatrix}x&v\\0&u\end{pmatrix},\begin{pmatrix}y&0\\0&1\end{pmatrix}\right]=
\begin{pmatrix}\left[x,y\right]&u^{-1}(v-xyx^{-1}v)\\0&1\end{pmatrix}.
\end{equation}
Here we view $v$ as a column vector.
\end{enumerate}
We will also make use of the following simple lemma concerning unimodular elements.
\begin{lem}\label{easylem}
If $(a_1,\dots,a_n)\in A^n$ is unimodular, then so is $(a_1,\dots,a_{n-1},a^2_n)$.
\end{lem}
\begin{proof}
Let $(t_1,\dots,t_n)$ so that
\begin{equation}\label{unim}
t_1a_1+\dots +t_na_n=1.
\end{equation}
Multiplying \eqref{unim} by $a_n$ we get
$$
a_nt_1a_1+\dots+a_nt_{n-1}a_{n-1}+t_na_n^2=a_n.
$$
Hence $a_n$ belongs to the ideal generated by $a_1,\dots,a_n^2$ and we get
$$
a_1A+\dots+a_{n-1}A+a_n^2A=a_1A+\dots+a_nA=A.
$$
\end{proof}
We turn to the proof of Theorem \ref{mainwidth}. We divide the proof in several steps.
The first step is reducing $\sigma$
to a copy of the semi-direct product $\Aff_{n-1}(A)=\SL_{n-1}(A)\ltimes A^{n-1}$. We consider two copies of
$\Aff_{n-1}(A)$ inside $\SL_n(A)$; the images of the embeddings
$(\gamma,v)\mapsto \begin{pmatrix}\gamma & v\\0&1\end{pmatrix}$, and
$(\gamma,v) \mapsto \begin{pmatrix}1 & v^T\\0&\gamma\end{pmatrix}$.
We denote these copies by $G_1$ and $G_2$ respectively.
\begin{lem}\label{step1}
Any non-central element $\sigma\in G$ can be reduced to a non-central element
of either $G_1$ or $G_2$, using at most $4$ $\q$-operations.
\end{lem}

\begin{proof}
Let $\alpha=\begin{pmatrix}a_1\\\vdots \\a_n\end{pmatrix}$
be the first column of $\sigma$, and $\beta=\begin{pmatrix}b_1,\dots,b_n\end{pmatrix}$
be the second row of $\sigma^{-1}$. Let
$\tau=I_n+pe_{12}$,
where $p$ is some arbitrary non-zero element in $\q$.

{\bf{Case 1.}} First, we consider the case where $\tau$ and $\sigma$ commute. Then we have
\begin{align*}
I_n&=[\sigma,\tau]\\
&=\sigma(I_n+pe_{12})\sigma^{-1}(I_n-pe_{12})\\
&=(I_n+p\sigma e_{12}\sigma^{-1})(I_n-pe_{12})\\
&=(I_n+p\alpha\beta)(I_n-pe_{12}),
\end{align*}
and so $\alpha\beta=e_{12}$. Hence
$\alpha=ue_1$ for some unit $u\in A$.
Thus in this case we have
$$
\sigma=\begin{pmatrix}u&v\\0&\sigma'\end{pmatrix},
$$
where $v\in\q^{n-1}$ and $\sigma'\in\GL(n-1,A,\q)$. Suppose that $\sigma'$
does not belong to the centralizer of $\E(n-1,A,\q)$. Then there exists
$\sigma''\in\E(n-1,\q,A)$ so that $[\sigma',\sigma'']\ne I_{n-1}$. By commutating
$\sigma$ with $\begin{pmatrix}1&0\\0&\sigma''\end{pmatrix}$ we then get
$$
\left[\sigma,\begin{pmatrix}1&0\\0&\sigma''\end{pmatrix}\right]=
\begin{pmatrix}1&*\\0&\left[\sigma',\sigma''\right]\end{pmatrix},
$$
which is non-central, due to the fact that $[\sigma',\sigma'']\ne I_{n-1}$, and we are done.
If $\sigma'$ belongs to the centralizer of $\E(n-1,A,\q)$, then
 $\sigma'=tI_{n-1}$ for some unit $t\in A$, and we have
$$
\begin{pmatrix}u&pte_1+v\\0&tI_{n-1}\end{pmatrix}=
\tau\sigma=\sigma\tau=
\begin{pmatrix}u&pue_1+v\\0&tI_{n-1}\end{pmatrix},
$$
and so $u=t$.
Thus in this case $\sigma=\begin{pmatrix}u&v\\0&uI_{n-1}\end{pmatrix}$.
Since $\sigma$ is non-central we must have $v\ne0$.
Thus there exists an element $\sigma''\in\E(n-1,A,\q)$ so that $v\sigma''^{-1}\ne v$.
Commutating $\sigma$ with $\begin{pmatrix}1&0\\0&\sigma''\end{pmatrix}$ we get,
 using \eqref{fact2},
$$
\left[\sigma,\begin{pmatrix}1&0\\0&\sigma''\end{pmatrix}\right]=
\begin{pmatrix}1&u^{-1}(v-v\sigma''^{-1})\\0&I_{n-1}\end{pmatrix},
$$
which is non-central because $v\sigma''^{-1}\ne v$.
Thus we see that if $\sigma$ and $\tau$ commute then $\sigma$
can be reduced to an element belonging to one of
the copies of the semi-direct product by at most two
$\q$-operations.

{\bf{Case 2.}} Assume now that $\sigma$ and $\tau$ do not commute, i.e.
\begin{equation}\label{eq5}
[\sigma,\tau]\ne I_n.
\end{equation}
We first reduce to the case where $(a_1,\dots,a_{n-1})\in A^{n-1}$ is unimodular.
It is clear that $\alpha$ is unimodular, and so
$(a_1,\dots,a_{n-1},a^2_n)$ is unimodular, due to Lemma \ref{easylem}.
Thus there exists $t=(t_1,\dots,t_{n-1})\in A^{n-1}$ so that
$(a_1+t_1a^2_n,\dots,a_{n-1}+t_{n-1}a^2_n)$ is unimodular.
Conjugating $\sigma$ by $\begin{pmatrix} I_{n-1}&a_nt\\0&1\end{pmatrix}$ (viewing $t$ as a column vector) we get
$$
\begin{pmatrix}I_{n-1}&a_nt\\0&1\end{pmatrix}\sigma\begin{pmatrix}I_{n-1}&-a_nt\\0&1\end{pmatrix}=
\begin{pmatrix}a_1+t_1a^2_n&*&\dots&*\\a_2+t_2a^2_n&*&\dots&*\\\vdots&\vdots&&\vdots\\a_{n-1}+t_{n-1}a^2_n&*&\dots&*\\**&*&\dots&*\end{pmatrix}.
$$
Thus after one conjugation we may assume that $(a_1,\dots,a_{n-1})$ is unimodular.
Assuming this, let $d=(d_1,\dots,d_{n-1})\in A^{n-1}$ such that $\sum_{k=1}^{n-1}d_ka_k-1=0$,
and let
$\lambda=\begin{pmatrix}I_{n-1}&0\\a_nd&1\end{pmatrix}$.
Conjugating $[\sigma,\tau]$ by $\lambda$ we get
\begin{equation}\label{eq6}
\begin{split}
\lambda^{-1}[\sigma,\tau]\lambda&=
\lambda^{-1}(\sigma\tau\sigma^{-1}\tau^{-1})\lambda\\
&=\lambda^{-1}(\sigma(I_n+pe_{12})\sigma^{-1}\tau^{-1})\lambda\\
&=\lambda^{-1}(I_n+p\alpha\beta)\tau^{-1})\lambda\\
&=\lambda^{-1}(\tau^{-1}+p\alpha\beta\tau^{-1})\lambda\\
&=\lambda^{-1}\tau^{-1}\lambda+p\lambda^{-1}\alpha\beta\tau^{-1}\lambda.
\end{split}
\end{equation}
Note that by construction, $\lambda^{-1}\alpha$ has $0$ in the last entry, so that
the last row of $\lambda^{-1}\alpha\beta\tau^{-1}\lambda$
is $0$.
Thus using \eqref{eq6} we get that the last row of $\lambda^{-1}[\sigma,\tau]\lambda$ is equal to the last row of
$\lambda^{-1}\tau\lambda$,
which is $(0,-pa_nd_1,0,\dots,0,1)$.
Denote $\rho=\lambda^{-1}[\sigma,\tau]\lambda$.
We claim that $\rho$ is non-central.
Indeed, otherwise $\rho$ would be a scalar matrix, and since (at least) one entry of $\rho$ is $1$, this would mean
$\rho=I_n$. But then $[\sigma,\tau]=I_n$, which contradicts \eqref{eq5}.
Now let $\alpha',\beta'$ be the first column and second row of $\rho$ and $\rho^{-1}$ respectively.
As before, if $[\rho,\tau]=I_n$ then $\alpha'=ue_1$
for some unit $u\in A$, and we are done.
Thus we may assume that
\begin{equation}\label{eq8}
[\rho,\tau]\ne I_n.
\end{equation}
As in \eqref{eq6}, we have
\begin{equation}\label{eq9}
[\rho,\tau]=\tau^{-1}+p\alpha'\beta'\tau^{-1}.
\end{equation}
As before, since $\alpha'$ has $0$ in the last entry, the last row of $\alpha'\beta'\tau^{-1}$ is $0$,
and \eqref{eq9} implies that
$[\rho,\tau]$ has the same last row as $\tau^{-1}$. But the last row of $\tau^{-1}$ is
$(0,\dots,0,1)$, so that $[\rho,\tau]$ belongs to $G_1$.
As before, $[\rho,\tau]$ is non-central since one of its diagonal entries equals $1$.
Thus we see that $\sigma$
can be reduced to an element belonging to one of
the copies of the semi-direct product by at most $4$
$\q$-operations.
\end{proof}

\begin{lem}\label{step2}
Any non-central matrix in $G_1$ may be reduced, using at most $1$ $\q$-operation
to a non-central matrix of the form $\begin{pmatrix}I_{n-1}&*\\0&1\end{pmatrix}$.
Any non-central matrix in $G_2$ may be reduced, using at most $1$ $\q$-operation
to a non-central matrix of the form  $\begin{pmatrix}I_{n-1}&0\\**&1\end{pmatrix}$.
\end{lem}
\begin{proof}
Let $\sigma=\begin{pmatrix}\gamma&v\\0&1\end{pmatrix}$,
where $\gamma\in\SL_{n-1}(A)$ and $v\in A^{n-1}$. Assume $\sigma$ is non-central.
If $\gamma=I_{n-1}$ there is nothing to prove. Otherwise there exists $v'\in \q^{n-1}$ such that
$\gamma v'\ne v'$. Let $\lambda=\begin{pmatrix}I_{n-1}&v'\\0&1\end{pmatrix}$.
Then $[\sigma,\lambda]=\begin{pmatrix}I_{n-1}&\gamma v' -v'\\0&1\end{pmatrix}$ is a non-central
element of the required form. The second part follows by a symmetric argument.
\end{proof}

\begin{lem}\label{step3}
Any non-central matrix of the form $\begin{pmatrix}I_{n-1}&*\\0&1\end{pmatrix}$
may be reduced to a non-trivial elementary matrix
using at most $1$ $\q$-operation.
The same is true for non-central elements of the form
$\begin{pmatrix}I_{n-1}&0\\**&1\end{pmatrix}$.
\end{lem}
\begin{proof}
Let $\sigma=\begin{pmatrix}I_{n-1}&v\\0&1\end{pmatrix}$, where
$v=\begin{pmatrix}v_1\\\vdots\\v_{n-1}\end{pmatrix}$.
Assume that $\sigma$ is non-central, so that $v_j\ne 0$ for some $1\le j\le n-1$. Let
$0\ne q\in\q$ be any non-zero element of $\q$.
Let $\lambda=\begin{pmatrix}I_{n-1}+qe_{1j}&0\\0&1\end{pmatrix}$. We have
$[\sigma,\lambda]=\begin{pmatrix}I_{n-1}&-qv_je_1\\0&1\end{pmatrix}$, which proves the first part
of the lemma. The second part follows by a symmetric argument.
\end{proof}

\begin{lem}\label{step4}
Let $1\le i\ne j\le n$. Then any non-trivial elementary matrix may be reduced to a non-trivial element
in $\E_{ij}(A)$ by using at most $3$ $\q$-operations.
\end{lem}
\begin{proof}
For the proof we use the relations
$$
[I_n+ae_{\alpha\beta},I_n+be_{\beta\gamma}]=I_n+abe_{\alpha\gamma},
$$
whenever $1\le\alpha,\beta,\gamma\le n$ are pairwise distinct.
Suppose we are given an element $I_n+re_{kl}\in \E_{kl}$ where $(i,j)\ne(k,l)$.
We consider the following two cases:
\begin{enumerate}
\item [Case 1] $i=k$ or $j=l$.
\item [Case 2] $i\ne k$, $j\ne l$.
\end{enumerate}
Suppose first we are in Case 1. If $j=l$, then $i,k,l$ are pairwise distinct, and thus
for any $q\in\q$
$$
[I_n+qe_{ik},I_n+re_{kl}]=I_n+qre_{il}=I_n+qre_{ij},
$$
and we are done. The case where $i=k$ is handled similarly.
Suppose we are in Case 2.
If $k,l,j$ are pairwise distinct, then for any $q\in\q$
$$
[I_n+re_{kl},I_n+qe_{lj}]=I_n+rqe_{kj},
$$
and we are in Case 1 again. If $k,l,j$ are not pairwise distinct then
necessarily $k=j$.
Since $n\ge 3$ we can find some $1\le h\le n$
such that $h,k,l$ are distinct and we
have that for any $q\in\q$
$$
[I_n+qe_{hk},I_n+re_{kl}]=I_n+qre_{hl}.
$$
If $h=i$ we are in Case 1 again. Otherwise commutate with
$I_n+qe_{lj}$ to get
$$
[I_n+qre_{hl},I_n+qe_{lj}]=I_n+rq^2e_{hj},
$$
and apply Case 1.
\end{proof}
\begin{proof}[Proof of Theorem \ref{mainwidth}]
The theorem follows immediately by applying Lemmas \ref{step1}, \ref{step2}, \ref{step3} and \ref{step4}
above.
\end{proof}

\begin{lem}\label{se4}
Let $A$ be a commutative ring with unit and $\q$ be an ideal in $A$.
Let $\sigma=\begin{pmatrix}a&b\\c&d\end{pmatrix}\in \Gamma(\q)$.
Suppose there exists a unit $u$ in $A$, such that $u\equiv1\mod c^2A$.
Then there exists an non-trivial
element of $\E_{12}(\q)$ which can be written as a product of at most $4$ conjugates of
$\sigma$ and $\sigma^{-1}$ by elements of $\Gamma(\q)$. The same is true if we replace $\E_{12}(\q)$
by $\E_{21}(\q)$ and require the existence of a unit $u$ with $u^2\equiv1\mod bA$.
\end{lem}
\begin{proof}
Let $x\in A$ such that $u^4=1+cx$ and let $t=ax$.
Let $S=\begin{pmatrix}1&t\\0&1\end{pmatrix}\sigma\begin{pmatrix}1&-t\\0&1\end{pmatrix}$, and $T=\begin{pmatrix}u^2&0\\0&u^{-2}\end{pmatrix}\sigma\begin{pmatrix}u^{-2}&0\\0&u^{2}\end{pmatrix}$.
Then both $S$ and $T$ are conjugates of $\sigma$ by elements of $\Gamma(\q)$. Let $Y=S^{-1}T$. Then
a direct computation shows that $Y=\begin{pmatrix}u^{-4}&q\\0&u^4\end{pmatrix}$, where $q\in\q$.
After commutating with a suitable element of the form $\begin{pmatrix}u^4&*\\0&u^{-4}\end{pmatrix}$
we get the desired element of $\E_{12}(\q)$. The other part of the lemma is obtained by using a symmetric argument.
\end{proof}

{}

\sc{
Leonid Polterovich, School of Mathematical Sciences, Tel Aviv University,
Tel Aviv 69978, Israel
\\

Yehuda Shalom,
School of Mathematical Sciences,
Tel Aviv University,
Tel Aviv 69978, Israel
\\

Zvi Shem-Tov,
Einstein Institute of Mathematics,
The Hebrew University of Jerusalem,
Jerusalem 91904, Israel}


\begin{thebibliography}{}
\bibitem{Bass}\label{Bass}
H. Bass, \emph{Algebraic $K$-theory}, Benjamin, New York, 1968.

\bibitem{BD}  S. Bochner, and D. Montgomery, \emph{Groups of differentiable and real or complex analytic transformations}, Ann. of Math. 46 (1945), 685 -- 694.

\bibitem {Bour}
N. Bourbaki, 
\emph{Lie groups and Lie algebras, Chapters 1--3}, Springer, Berlin (1998). 


\bibitem{BHW}\label{BHW}
J. Bowden, S. Hensel, R. Webb,
\emph{Quasi-morphisms on surface diffeomorphism groups}, preprint arXiv:1909.07164 (2019).

\bibitem{BO} L. Buhovsky and Y. Ostrover,
\emph{On the uniqueness of Hofer's geometry},
Geom. Funct. Anal. 21 (2011), no. 6, 1296 --1330.

\bibitem{BIP} D. Burago, S. Ivanov, L. Polterovich, {\it Conjugation-invariant norms on groups of geometric origin}, Groups of diffeomorphisms 221–250,  Adv. Stud. Pure Math. 52, Math. Soc. Japan, Tokyo (2008).


\bibitem{Thom}
P. A. Dowerk, A. Thom, {\it{ Bounded normal generation and invariant automatic continuity}}, Adv. Math. (2019) 124–169.

\bibitem{Hof2}
H. Hofer, {\it{On the topological properties of symplectic maps.}}, Proceedings of the Royal Society of Edinburgh Section A: Mathematics. (1990);115(1-2):25-38.

\bibitem{Igl} P. Igl\'{e}sias,
\emph{Les ${\rm SO}(3)$-vari\'{e}tés symplectiques et leur classification en dimension $4$},
Bull. Soc. Math. France 119 (1991), no. 3, 371 -- 396.

\bibitem{Kechris} A. Kechris,  \emph{Classical Descriptive Set Theory}, Springer, 2012.

\bibitem{Ked}
J. Kedra, A. Libman, B. Martin, {\it Strong and uniform boundedness of groups},
 Journal of Topology and Analysis (2021)
\bibitem{KL}
V. Klee,
{\it{Invariant metrics in groups (solution of a problem of Banach)}},
Proc. Amer. Math. Soc. 3 (1952), 484–487.

\bibitem{Mont}
D. Montgomery, L. Zippin, {\it {Topological transformation groups}},
Interscience Publishers, New York-London (1955), xi+282 pp.

\bibitem{Moore}
C. Moore, {\it{Group extensions of $p$-adic and adelic linear groups}},
Inst. Hautes Études Sci. Publ. Math. No. 35 (1968), 157–222.
\bibitem{Nica}
B. Nica, {\it {On bounded elementary generation for $\SL_n$ over polynomial rings}},
Isr. J. Math. 225, 403–410 (2018)

\bibitem{Nienhuys}
J. Nienhuys, {\it{Not locally compact monothetic groups. I}}, Indagationes Mathematicae (Proceedings). 73 (1970), 295-310.

\bibitem{NIK}\label{NIK}
N. Nikolov, Dan. Segal,
{\it{On finitely generated profinite groups. I. Strong completeness and uniform bounds}},
Ann. of Math. (2) 165 (2007), no. 1, 171–238.

\bibitem{NIK3}\label{NIK3}
N. Nikolov, D. Segal,
{\it{Generators and commutators in finite groups; abstract quotients of compact groups}},
Invent. Math. 190 (2012), no. 3, 513–602.

\bibitem{Ros} 
C. Rosendal {\it{Automatic Continuity of Group Homomorphisms}}, Bulletin of Symbolic Logic 15,
(2009), no. 2, 184-214

\bibitem{SerreBook} 
J. P Serre {\it{Lie algebras and Lie groups: 1964 lectures given at Harvard University}}, Springer, 2009.


\bibitem{SW}\label{SW}
Y. Shalom, G. Willis,
{\it{Commensurated subgroups of arithmetic groups, totally disconnected groups and adelic rigidity}},
Geom. Funct. Anal. 23 (2013), no. 5, 1631–1683.


\bibitem{Sh}
Y. Shalom, {\it{Bounded generation and Kazhdan's property (T)}},
Inst. Hautes Études Sci. Publ. Math. No. 90 (1999), 145–168.

\bibitem{Stepanov}
A. Stepanov, 
{\it {Structure of Chevalley groups over rings via universal localization}},
Journal of Algebra 450 (2016), 522-548. 

\bibitem{Trost}
A. Trost, {\it{Strong boundedness of split Chevalley groups}}, to appear in Israel Journal of Mathematics (2021). 

\bibitem{Trost3}
A. Trost {\it{The dichotomy property of $\SL_2 (R)$-A short note}}, arXiv:2202.08319. 

\bibitem{trost2}
A. Trost, {\it{Bounded generation by root elements for Chevalley groups defined over rings
of integers of function fields with an application in strong boundedness}}, arXiv:
2108.12254 [math.GR].


\bibitem{Ts1}
T. Tsuboi,
{\it{On the uniform simplicity of diffeomorphism groups}},
Differential geometry, World Sci. Publ., Hackensack, NJ (2009), 43–55.

\bibitem{Ts2}
T. Tsuboi,
{\it{On the uniform perfectness of the groups of diffeomorphisms of even-dimensional manifolds}},
Comment. Math. Helv. 87 (2012), 141–185


\bibitem{DW}
D. Witte Morris,
{\it{Bounded generation of $\SL(n,A)$ (after D. Carter, G. Keller, and E. Paige)}},
New York J. Math. 13 (2007), 383–421.

\bibitem{DW2}
D. Witte Morris,
{\it{Introduction to arithmetic groups}}, 

http://arxiv.org/abs/math/0106063


\end{thebibliography}
\end{document}